\documentclass{article}
\usepackage[margin=1.2in]{geometry}
\usepackage{amsmath}
\usepackage{amsfonts}
\usepackage{amsthm}
\usepackage{amssymb}
\usepackage{bbm}
\usepackage[plain]{algorithm}
\usepackage{algorithmicx}
\usepackage{placeins}
\usepackage[style=3]{mdframed}
\mdfsetup{pstrickssetting={linestyle=dashed}}
\usepackage{xcolor}
\usepackage{hyperref}
\usepackage{titling}
\usepackage{graphicx}
\usepackage{enumitem}
\usepackage{authblk}

\hypersetup{
    colorlinks,
    linkcolor={red!50!black},
    citecolor={blue!50!black},
    urlcolor={blue!80!black}
}

\newcommand{\R}{\ensuremath{\mathbb{R}}}

\newfont{\BB}{msbm10 scaled\magstep1}
\newfont{\bb}{msbm8}
\newcommand{\cS}{{\mathcal S}}
\newcommand{\cK}{{\mathcal K}}

\newcommand{\C}{{\mathbb C}}
\newcommand{\cE}{{\mathcal E}}

\newcommand{\E}[0]{\mathbbm{E}}
\newcommand{\N}[0]{\mathbbm{N}}
\newcommand{\1}[0]{\mathbbm 1}

\newtheorem{theorem}{Theorem}

\newtheorem{problem}{Problem}
\newtheorem{definition}{Definition}
\newtheorem{proposition}{Proposition}
\newtheorem{example}{Example}

\newtheorem*{remark}{Remark}

\newcommand{\I}{\mathfrak{I}}

\renewcommand{\phi}{\varphi}

\def\D{\mathbb{D}}
\newcommand{\De}{\mathrm{d}}
\newtheorem*{claim}{Claim}

\newcommand*{\myproofname}{Proof}
\newenvironment{proofclaim}[1][\myproofname]{\begin{proof}[#1]}{\end{proof}}
\newtheorem{theoremR}{Theorem}

\newtheorem{innercustomlemma}{}
\newenvironment{customlemma}[1]
  {\renewcommand\theinnercustomlemma{#1}\innercustomlemma}
  {\endinnercustomlemma}

\newtheorem{innercustomobs}{}

\title{
Traversing the Schr\"odinger Bridge strait:\\ Robert Fortet's marvelous proof redux}

\author[1]{Montacer Essid}
\author[2]{Michele Pavon}

\affil[1]{Courant Institute of Mathematical Sciences, NYU}
\affil[2]{Dipartimento di Matematica ``Tullio Levi-Civita", Universit\`a di Padova}

\date{\today}

\begin{document}
\maketitle

\begin{abstract} In the early 1930's, Erwin Schr\"odinger, motivated by his quest for a more classical formulation of quantum mechanics, posed a large deviation problem for a cloud of independent Brownian particles. He showed that the solution to the problem could be obtained trough a system of two linear equations with nonlinear coupling at the boundary ({\em Schr\"odinger system}). Existence and uniqueness for such a system, which represents a sort of bottleneck for the problem,  was first established by R. Fortet in 1938/40 under rather general assumptions by proving convergence of an ingenious but complex approximation method. It is the first  proof of what  are nowadays called Sinkhorn-type algorithms in the much more challenging continuous case. Schr\"odinger bridges are also an early example of the maximum entropy approach and have been more recently recognized as a regularization of the important Optimal Mass Transport problem.

Unfortunately, Fortet's contribution is by and large ignored in contemporary literature. This is likely due to the complexity of his approach coupled with an idiosyncratic exposition style and to missing details and steps in the proofs. Nevertheless, Fortet's approach maintains its importance to this day as it provides the only existing algorithmic proof, in the continuous setting, under rather mild assumptions. It can be adapted, in principle, to other relevant optimal transport problems.
It is the purpose of this paper to remedy this situation by rewriting the bulk of his paper with all the missing passages and in a transparent fashion so as to make it fully available to the scientific community. We consider the problem in $\R^d$ rather than $\R$ and use as much as possible his notation to facilitate comparison. 
\end{abstract}

\section{Introduction}\label{intro}
In 1931/32, Erwin Schr\"odinger showed that the solution to a hot gas experiment (large deviations problem) could be reduced to establishing existence and uniqueness of a pair  of positive functions $(\varphi,\hat{\varphi})$ satisfying what was later named the Schr\"odinger system, see (\ref{eq:Schronestep}) below. This is a system of two linear PDE's with nonlinear coupling at the boundary. Besides Schr\"odinger's original motivation, this problem features two more:  The first is a {\em maximum entropy principle in statistical inference}, namely choosing a posterior distribution so as to make the fewest number of assumptions about what is beyond the available information. This inference method has been noticeably developed over the years by Jaynes, Burg, Dempster and Csisz\'{a}r \cite{Jaynes57,Jaynes82,BURG1,BLW,Dem,csiszar0,csiszar1,csiszar2}. The second, more recent, is regularization of the Optimal Mass Transport problem \cite{Mik,mt1,MT,leo,leo2,CGP2016,PC} providing an effective computational approach to the latter; see, e.g., \cite{Cuturi,BCCNP,CGP5,CGP9}.

 The first proof of existence and uniqueness for the Schr\"odinger system was provided in 1938/40 by the French analyst Robert Fortet \cite{For0,For}. Subsequent significant contributions are due to Beurlin (1960), Jamison (1975), Zambrini (1986) and F\"ollmer (1988) \cite{Beu,Jam2,Zam,F2}. Fortet's proof is algorithmic, being based on a complex iterative scheme. It represents also the first proof, in the much more challenging continuous setting, of convergence of a procedure (called {\em iterative proportional fitting} (IPF)) proposed by  Deming and Stephan \cite{DS1940} (1940) for contingency tables. In the latter discrete setting, the first convergence proof was provided in a special case  some twenty five years after Fortet and Deming-Stephan by R. Sinkhorn \cite{Sin64}, who was unaware of their work. These iterative schemes are nowadays often called {\em Sinkhorn-type algorithms} or {\em Iterative Bregman projections}; cf., e.g., \cite{Cuturi,BCCNP,CPSV}.   Unfortunately, in spite of its importance, Fortet's contribution has by and large sunk into oblivion. This is arguably  due to the complexity of his approach, to the unconventional organization of the paper and to a number of gaps in his arguments. 
 Nonetheless, to this day, Fortet's existence result is the central one as it is based on the convergence of an algorithm under rather weak assumptions and does not require a kernel bounded away from zero. Other proofs in the continuous setting \cite{Beu,Jam2,F2}, \cite[Section 2]{leo} are non constructive except  \cite{CGP9}. The latter proof, however, assumes compactly supported marginal distributions. Finally, Fortet's approach may, in principle, be taylored to attack  other significant optimal transport problems. 

 The purpose of this paper is to make his fundamental contribution fully available to the scientific community. To achieve this, we review, elaborate upon and generalize to $\R^d$ Fortet's proof of existence and uniqueness for the Schr\"odinger system. We systematically fill in all the missing steps and provide thorough explanations of the rationale behind different articulations of his approach, but keep as much as possible his original notation to make comparison simpler. Nevertheless, we have chosen to reorganize the paper to improve its readability since, for instance, Fortet often presents the proof before the statement of the result.  Finally, our original work, completing a sketchty proof, or proving Fortet's claims or making explicit what is implicit in \cite{For}, appears in a sequence of Propositions, Observations and one Claim (all not present in \cite{For}) to make it easily identifiable.
 
 The paper is organized as follows: In the next two sections, we provide a concise introduction to the Schr\"odinger bridge problem which is not present  in \cite{For}. We include, for the benefit of the reader, Schr\"odinger's original motivation, elements of the transformation of the large deviation problem into a maximum entropy problem and a derivation of the Schr\"odinger system. Section \ref{statement} features Fortet's statement of the problem and his basic assumptions. Section \ref{Existence I} is devoted to his first existence theorem. In Section \ref{Existence II}, a special case of his second existence theorem is stated and his uniqueness result is proved. In Section \ref{Hilbert}, Fortet's approach is compared with the one based on contracting the Hilbert Metric. The conclusions section discusses relation of his method  to subsequent work and an outlook on potential extensions, interpretations and applications of Fortet's approach.

\section{The Hot Gas Gedankenexperiment}
In 1931-32, Erwin Schr\"{o}dinger considered the following thought experiment \cite{S1,S2}: A cloud of $N$ independent Brownian particles is evolving in time in $\R^3$. Suppose that at $t=0$ the empirical distribution is $\rho_0(x)dx$ and at $t=1$ it is  $\rho_1(x)dx$. 
If $N$ is large, say of the order of Avogadro's number, we expect, by the law of large numbers,
$$\rho_1(y)\approx \int_{\R^3}p(0,x,1,y)\rho_0(x)dx,
$$
where
\begin{equation}\label{transitiondensity}p(s,y,t,x)=\left[2\pi(t-s)\right]
^{-\frac{3}{2}}\exp\left[-\frac{|x-y|^2} {2(t-s)}\right],\quad s<t
\end{equation}
is the transition density of the Wiener process. If this is not the case, the particles have been transported in an unlikely way. But of the many unlikely ways in which this could have happened, which one is
the most likely? In modern probabilistic terms, this is a problem of {\em large deviations of the empirical distribution} as observed by F\"ollmer \cite{F2}. 
The area of large deviations is concerned with the probabilities of very rare events. Thanks to Sanov's theorem \cite{SANOV}, Schr\"odinger's problem can be turned into a maximum entropy problem for distributions on trajectories. Let $\Omega=C([0,1];\R^d)$ be the space of $\R^d$ valued continuous functions and let $X^1, X^2,\ldots$ be i.i.d. Brownian evolutions on $[0,1]$ with values in $\R^d$  ($X_i$ is distributed according to the Wiener measure $W$ on $C([0,1];\R^d)$) with initial marginal $\rho_0(x)dx$ . The {\em empirical distribution} $\mu_N$ associated to $X^1,X^2,\ldots X^N$ is defined by
\begin{equation}\label{emp}\mu_N:=\frac{1}{N}\sum_{i=1}^N\delta_{X^i}.
\end{equation}
Notice that (\ref{emp}) defines a map from $\Omega^N$ to the space ${\cal D}$ of probability distributions on $C([0,1];\R^d)$. Hence, if $E$ is a subset of ${\cal D}$, it makes sense to consider $W^N(\mu_N\in E)$. By the law of large numbers for empirical measures, see e.g. \cite[Theorem 11.4.1]{dudley}, the distributions $\mu_N$ converge weakly
\footnote{\linespread{1}\small Let ${\cal V}$ be a metric space and ${\cal D}({\cal V})$ be the set of probability measures defined on ${\cal B}({\cal V})$, the Borel $\sigma$-field of ${\cal V}$. We say that a sequence $\{P_N\}$ of elements of ${\cal D}({\cal V})$ converges weakly to $P\in {\cal D}({\cal V})$, and write  $P_N\Rightarrow P$, if $\int_{\cal V}fdP_N\rightarrow\int_{\cal V}f dP$ for every bounded, continuous function $f$ on ${\cal V}$.} to $W$ as $N$ tends to infinity. Hence, if $W\not\in E$, we must have $W^N(\mu_N\in E)\searrow 0$. Large deviation theory, see e.g. \cite{ellis,DZ}, provides us with a much finer result: Such a decay is {\em exponential} and the exponent may be characterized solving a {\em maximum entropy problem}. Indeed, in our setting, let $E={\cal D}(\rho_0,\rho_1)$, namely distributions on $C([0,1];\R^d)$ having marginal densities $\rho_0$ and $\rho_1$ at times $t=0$ and $t=1$, respectively. Let 
\[ \D(P\| W) = \begin{cases} \E_{P}\left( \log \frac{\De P}{\De W} \right), \quad & \mbox{if $ P \ll W$},\\+ \infty \quad & \mbox{otherwise} \end{cases} \]
be the relative entropy functional or Kullback-Leibler divergence between $P$ and $W$.
Then, a consequence of Sanov's theorem,  asserts that if the "prior" $W$ does not have the required marginals, the sequence
$$W^N\left[\frac{1}{N}\sum_{i=1}^N\delta_{X^i}\in\cdot\right]
$$
satisfies a large deviation principle with rate function $\D(\cdot\|W)$. This is often abbreviated as follows:
The probability of observing an empirical distribution $\mu_N$ in ${\cal D}(\rho_0,\rho_1)$ decays according to 
$$W^N\left(\frac{1}{N}\sum_{i=1}^N\delta_{X^i}\in{\cal D}(\rho_0,\rho_1)\right)\sim\exp\left[-N \inf\left\{\D(P\|W); P\in{\cal D}(\rho_0,\rho_1)\right\} \right]. 
$$
Thus, the most likely random evolution between two given marginals is the solution of the Schr\"odinger Bridge Problem:
\begin{problem}\label{bridge}
\begin{equation}{\rm Minimize}\quad \D(P\|W) \quad {\rm over} \quad P\in{\cal D}(\rho_0,\rho_1).
\end{equation} 
\end{problem}
\noindent
The optimal solution is called the \emph{Schr\"odinger bridge} between $\rho_0$ and $\rho_1$ over $W$, and  its marginal flow $(\rho_t)$ is the \emph{entropic interpolation}.

Let $P\in\cal D$ be a finite-energy diffusion, namely under $P$ the canonical coordinate process $X_t(\omega)=\omega(t)$ has a (forward) Ito differential
\begin{equation}\label{fordiff}dX_t=\beta_tdt+dW_t
\end{equation}
 where $\beta_t$ is adapted to $\{{\cal F}^-_t\}$ (${\cal F}^-_t$ is the $\sigma$-algebra of events observable up to time $t$) and
 \begin{equation}\label{finenergy}\E_P\left[\int_0^1\|\beta_t\|^2dt\right]<\infty.
 \end{equation}
 Let
 \[P_x^y=P\left[\,\cdot\mid X_0=x,X_1=y\right],\quad W_x^y=W\left[\,\cdot\mid X_0=x,X_1=y\right]
\]
be the disintegrations of $P$ and $W$ with respect to the initial and final positions. Let also $\pi$ and  $\pi^{W}$
be the joint initial-final time distributions under $P$ and $W$, respectively. Then, we have the following decomposition of the relative entropy \cite{F2}
\begin{eqnarray}
&&\D(P\|W)=E_P\left[\log\frac{dP}{dW}\right]=\nonumber\\&&\iint\left[\log\frac{\pi(x,y)}{\pi^{W}(x,y)}\right]\pi(x,y)dxdy+\iint\left(\log\frac{dP^{y}_{x}}{dW^{y}_{x}}\right)dP^y_x \pi(x,y)dxdy.\label{decomposition}
\end{eqnarray}
Both terms are non-negative. We can make the second zero by choosing $P^{y}_{x}=W^{y}_{x}$.  Thus, the problem reduces to the static one
\begin{problem}\label{static}
Minimize over densities $\pi$ on $\R^d\times\R^d$ the index
\begin{equation}\label{staticindex}
\D(\pi\|\pi^{W})=\iint\left[\log\frac{\pi(x,y)}{\pi^{W}(x,y)}\right]\pi(x,y)dxdy
\end{equation}
subject to the (linear) constraints
\begin{equation}\label{constraint} \int \pi(x,y)dy=\rho_0(x),\quad \int \pi(x,y)dx=\rho_1(y).
\end{equation}
\end{problem}
If $\pi^*$ solves the above problem, then
$$P^*(\cdot)=\int_{\R^d\times\R^d} W_{xy}(\cdot)\pi^*(x,y)dxdy,
$$
solves Problem \ref{bridge}.

Consider now the case when the prior is $W_\epsilon$, namely Wiener measure with  variance $\epsilon$, so that
\[p(0,x,1,y)=\left[2\pi\epsilon\right]
^{-\frac{d}{2}}\exp\left[-\frac{|x-y|^2} {2\epsilon}\right].
\]
Using $\pi^{W_\epsilon}(x,y)=\rho_0(x)p(0,x;1,y)$ and the fact that the quantity
\[\iint\left[\log\rho_{0}(x)\right]\pi(x,y)dxdy=\int\left[\log\rho_{0}(x)\right]\rho_{0}(x)dx
\]
is independent of $\pi$ satisfying (\ref{constraint})\footnote{\linespread{1}\small The initial marginal of the prior measure, as long as  $\rho_0(x)dx$ is at finite relative entropy from it, does not play any role in the optimization problem. Instead of $\rho_0(x)dx$, which is the standard case in control problems, another popular choice is Lebesgue measure so that the prior is an unbounded measure called {\em stationary Wiener measure}, see e.g. \cite{leo}.}, we get
\begin{eqnarray}\nonumber
\D(\pi\|\pi^{W_\epsilon})&=&-\iint\left[\log\pi^{W_\epsilon}(x,y)\right]\pi(x,y)dxdy+\iint\left[\log\pi(x,y)\right]\pi(x,y)dxdy\\&=&\iint\frac{|x-y|^2}{2\epsilon}\pi(x,y)dxdy-{\cal S}(\pi)+C,
\end{eqnarray}
where ${\cal S}$ is the differential entropy and $C$ does not depend on $\pi$. Thus, Problem \ref{static} of minimizing $\D(\pi\|\pi^{W_\epsilon})$ over $\Pi(\rho_0,\rho_1)$, namely the  ``couplings" of $\rho_0$ and $\rho_1$\footnote{Probability densities on  $\R^n\times \R^n$  with marginals $\rho_0$ and $\rho_1$.}, is equivalent to
\begin{equation}\label{eq:regularizedOT}
 \inf_{\pi \in \Pi(\rho_0,\rho_1)}  \int\frac{|x-y|^2}{2} \pi(x,y) \De x\De y +  \epsilon \int\pi(x,y) \log \pi(x,y) \De x\De y,
\end{equation}
namely a regularization of Optimal Mass Transport (OMT) \cite{Vil} with quadratic cost function obtained by subtracting a term proportional to the entropy.

\section{Derivation of the Schr\"{o}dinger System}\label{SchrSystem}
We outline the derivation of the  Schr\"odinger system for the sake of continuity in exposition. Two good surveys on Schr\"odinger Bridges are \cite{Wak,leo}.
The Lagrangian function for Problem \ref{static} has the form
\begin{eqnarray}\nonumber&&{\cal L}(\pi;\lambda,\mu)= \iint\left[\log\frac{\pi(x,y)}{\pi^{W}(x,y)}\right]\pi(x,y)dxdy\\&&+\int \lambda(x)\left[\int \pi(x,y)dy-\rho_0(x)\right]+\int \mu(y)\left[\int \pi(x,y)-\rho_1(y)\right].\nonumber
\end{eqnarray}
Setting the first variation with respect to $\pi$ equal to zero, we get the (sufficient) optimality condition
$$1+\log \pi^*(x,y)-\log p(0,x,1,y)-\log \rho_0(x)+\lambda(x)+\mu(y)=0,
$$
where we have used the expression $\pi^{W}(x,y)=\rho_0(x)p(0,x,1,y)$ with $p$ as in (\ref{transitiondensity}). We get
\begin{eqnarray*}
\frac{\pi^*(x,y)}{p(0,x,1,y)}&=&\exp\left[\log \rho_0(x)-1-\lambda(x)-\mu(y)\right]\\
&=&\exp\left[\log \rho_0(x)-1-\lambda(x)\right]\exp\left[-\mu(y)\right].
\end{eqnarray*}
Hence,  the ratio $\pi^*(x,y)/p(0,x,1,y)$ factors into a function of $x$ times a function of $y$. Denoting these by $\hat{\varphi}(x)$ and $\varphi(y)$, respectively, we can then write the optimal $\pi^*(\cdot,\cdot)$ in  the form 
\begin{equation}\label{optimaljoint} \pi^*(x,y)=\hat{\varphi}(x) p(0,x,1,y)\varphi(y),
\end{equation} where $\varphi$ and $\hat{\varphi}$ {must satisfy}
\begin{eqnarray}\hat{\varphi}(x)\int p(0,x,1,y)\varphi(y)dy&=&\rho_0(x),\label{opt1}\\
\varphi(y)\int p(0,x,1,y)\hat{\varphi}(x)dx&=&\rho_1(y).\label{opt2}
\end{eqnarray}
Let us define $\hat{\varphi}(0,x)=\hat{\varphi}(x)$, $\quad \varphi(1,y)=\varphi(y)$ and  
$$\hat{\varphi}(1,y):=\int p(0,x,1,y)\hat{\varphi}(0,x)dx,\quad \varphi (0,x):=\int p(0,x,1,y)\varphi(1,y).
$$
Then, (\ref{opt1})-(\ref{opt2}) can be replaced by the system
\begin{subequations}\label{eq:Schronestep}
\begin{eqnarray}\label{eq:SchronestepA}
&&\hat{\varphi}(1,y)=\int p(0,x,1,y)\hat{\varphi}(0,x)dx,\\\label{SchronestepB}\quad &&\varphi (0,x)=\int p(0,x,1,y)\varphi(1,y)dy,\\\label{eq:SchronestepC}&&\varphi(0,x)\cdot\hat{\varphi}(0,x)=\rho_0(x),\\\label{eq:SchronestepD}&&\varphi(1,y)\cdot\hat{\varphi}(1,y)=\rho_1(y).
\end{eqnarray}
\end{subequations}
The arguments leading to (\ref{eq:Schronestep}) apply to the much more general case where the prior measure on path space is not Wiener measure but any finite energy diffusion measure $\bar{P}$ \cite{F2}. In that case, $p(0,x,1,y)$ is the transition density of $\bar{P}$. As already said, the question of existence and uniqueness of  positive functions $\hat{\varphi}$, $\varphi$ satisfying (\ref{eq:Schronestep}), left open by Schr\"{o}dinger, is a highly nontrivial one and was settled in various degrees of generality by Fortet, Beurlin, Jamison, F\"{o}llmer and L\'eonard \cite{For,Beu,Jam2,F2,leo}. The pair $(\varphi,\hat{\varphi})$ is unique up to multiplication of $\varphi$ by a positive constant $c$ and division of $\hat{\varphi}$ by the same constant. A proof based on convergence of an iterative scheme in Hilbert's projective metric (convergence of rays in a suitable cone) was provided in \cite{CGP9} in the case when both marginals have compact support.

At each time $t$, the marginal $\rho_t$ factorizes as
\begin{equation}\label{factorization}\rho_t(x)=\varphi(t,x)\cdot\hat{\varphi}(t,x).
\end{equation}
Schr\"odinger saw ``Merkw\"{u}rdige Analogien zur Quantenmechanik, die mir sehr des Hindenkens wert erscheinen"\footnote{Remarkable analogies to quantum mechanics which appear to me very worth of reflection.} Indeed (\ref{factorization}) resembles Born's relation
\[\rho_t(x)=\psi(t,x)\cdot\bar{\psi}(t,x)
\]
with $\psi$ and $\bar{\psi}$ satisfying two adjoint equations like $\varphi$ and $\hat{\varphi}$. Moreover, the solution of Problem \ref{bridge} exhibits the following remarkable {\em reversibility property}: Swapping the two marginal densities $\rho_0$ and $\rho_1$, the new solution is simply the time reversal of the previous one, cf. the title ``On the reversal of natural laws" of \cite{S1}.

We mention, for the benefit of the reader, that there exist also dynamic versions of the problem such as stochastic control formulations  originating with \cite{DP,DPP,PW,Mikami}. These formulations are particularly relevant in applications where the prior distribution on paths  is  associated to the uncontrolled ({\em free}) evolution of a dynamical system, see e.g \cite{CGP1,CGP3,CGPcooling} and in image morphing/interpolation \cite[Subsection 5.3]{CGP9}.
 The stochastic control problems leads directly to a fluid dynamic formulation, see \cite{leo,CGP2016}. The latter can be viewed as a regularization of the Benamou-Brenier dynamic formulation of Optimal Mass Transport \cite{BB}.

\section{Fortet's Statement of the Problem}\label{statement}
Let $d \in \N^*$. Define by $\mathcal{B}(\I)$ the Borel $\sigma$-algebra of $\I \subseteq \R^d$, and $m$ the Lebesgue measure on $\I$. Almost everywhere (a.e.) will always be intended with respect to $m$.  In this paper, measurable functions with respects to the Borel $\sigma$-algebra on their corresponding interval of definition will simply be referred to as measurable. 
Moreover, all properties concerning measures of sets will (tacitly) refer to their Lebesgue measure. From here on, we shall try to adhere to Fortet's notation as much as possible. In particular, with respect to the notation employed in Section \ref{intro}, the following changes are made: The two marginal densities $\rho_0(x)$ and $\rho_1(y)$ are replaced by $\omega_1(x)$ and $\omega_2(y)$, respectively. The kernel (transition density) $p(0,x,1,y)$ is replaced by $g(x,y)$. Finally, the pair $(\hat{\varphi}(x),\varphi(y))$ is replaced by the pair $(\phi(x),\psi(y))$.\\

Let $\I^1,\I^2 \subseteq \R^d$ be closed sets with non-empty interior, but not necessarily bounded.

Let $\omega_1: \I^1 \to \R,\omega_2: \I^2 \to \R $ and $g: \I^1 \times \I^2 \to \R$ satisfying the \textbf{assumptions (H)}:
%
\begin{enumerate}[label=(H.\roman*)]
\item \label{Hi} $g(x,y) \geq 0, \quad \forall x \in \I^1, \forall y \in \I^2$;
\item \label{Hii} $\omega_1(x) \geq 0,\; \omega_2(y) \geq 0, \quad \forall x \in \I^1,\forall y \in \I^2$;
\item \label{Hiii} ${\displaystyle \int_{\I^1} \omega_1(x)dx = \int_{\I^2} \omega_2(y) dy} = 1$;
\item \label{Hiv} $g$ is continuous;
\item \label{Hv} There exists $\Sigma>0$ such that $g(x,y)< \Sigma$, $\forall x \in \I^1, y \in \I^2$;
\item \label{Hvi} $\forall x \in \I^1$, $y \mapsto g(x,y)$ vanishes only on a set of measure $0$ in $\I^2$;
\item \label{Hvii} $\forall y \in \I^2$, $x \mapsto g(x,y)$ vanishes only on a set of measure $0$ in $\I^1$;
\item \label{Hviii} $\omega_1$ and $\omega_2$ are continuous.
\end{enumerate}
Notice that in Fortet's paper, \ref{Hi}-\ref{Hiii} are denoted Hypothesis I \cite[p.83]{For}, whereas hypotheses \ref{Hiv}-\ref{Hviii} are called Hypothesis II a) and b) \cite[p.85]{For}.

We are seeking a solution $(\phi,\psi)$ of the following Schr\"odinger system of equations (S):
\begin{equation} \label{SchrodingerSystem}
\begin{cases}
{\displaystyle \phi(x) \int_{\I^2} g(x,y)\psi(y) dy} = \omega_1(x), \\\\
{\displaystyle \psi(y) \int_{\I^1} g(x,y)\phi(x) dx} = \omega_2(y),
\end{cases} \tag{S}
\end{equation}
cf. system \eqref{opt1}-\eqref{opt2}.

\section{First Existence Theorem}\label{Existence I}
\subsection{Theorem I}
\begin{theoremR} {\em \cite[p.96]{For}}\label{thm:ex1}
Assume \emph{(H)}, as well as the condition:
\begin{equation}\label{cond:thm1}
\int_{\I^2} \frac{\omega_2(y)}{ \left[{\displaystyle \int_{\I^1} g(z,y) \omega_1(z) dz}\right] } dy < +\infty \tag{$\star$}
\end{equation}

Then:
\begin{enumerate}[label=\roman*)]
\item System \eqref{SchrodingerSystem} admits a solution $(\phi,\psi)$;
\item $\phi$ is non-negative and continuous;
\item $\phi$ vanishes only if  $x \in \I^1$ is such that $\omega_1(x) = 0$;
\item $\psi$ is measurable and non-negative;
\item $\psi$ vanishes, up to a zero measure set, only for $y \in \I^2$ such that $\omega_2(y) = 0$. 
\end{enumerate}
\end{theoremR}
\subsection{Application: the Bernstein Case}
Consider the case where $\I^1 = \I^2 = \R$ ($d=1$), and we have Gaussian marginals and transition kernel:
\[
\omega_1(x) = \frac{1}{\sqrt{2 \pi \sigma_1^2}} e^{-x^2/2\sigma_1^2}, \quad \omega_2(y) = \frac{1}{\sqrt{2 \pi \sigma_2^2}} e^{-y^2/2\sigma_2^2}, \quad g(x,y) = \frac{1}{\sqrt{2 \pi \sigma^2}} e^{-(y-x)^2/2 \sigma^2}
\]
for $\sigma_1,\sigma_2, \sigma >0$.

Then the integrand in \eqref{cond:thm1} is:
\[
\frac{\omega_2(y)}{ \left[{\displaystyle \int_{\R} g(z,y) \omega_1(z) dz}\right] } = \frac{\sqrt{\sigma^2 + \sigma_1^2}}{\sigma_2^2} e^{- y^2 \frac{\sigma^2+\sigma_1^2 - \sigma_2^2}{\sigma_2(\sigma_1+\sigma)} }
\]
which is integrable if and only if $\sigma^2+\sigma_1^2 - \sigma_2^2 > 0$. If $\sigma_1 \geq \sigma_2$, this is true and one can apply Theorem \ref{thm:ex1}. If it is not the case, exchange the roles of $\omega_1$ and $\omega_2$ to satisfy condition \eqref{cond:thm1}, and apply the theorem.
Hence up to exchanging the marginals, one can always show existence of a solution to the system \eqref{SchrodingerSystem} in the Bernstein case.

Consider now the case $\I^1 = \I^2 = \R^d$, $d>1$, and   
\[
\omega_i(x) = \frac{1}{(2 \pi |\Sigma_i|)^{d/2}} e^{-x^T \Sigma_i^{-1} x/2}, \quad i=1,2, \quad g(x,y) = \frac{1}{(2 \pi |\Sigma|)^{d/2}} e^{-(y-x)^T \Sigma^{-1} (y-x)/2}
\]
for some symmetric, positive definite matrices $\Sigma,\Sigma_1,\Sigma_2$. Then 
\[
\frac{\omega_2(y)}{ \left[{\displaystyle \int_{\R} g(z,y) \omega_1(z) dz}\right] } = \frac{|\Sigma+\Sigma_1|^{d/2}}{|\Sigma_2|^{d/2}} e^{- y^T [\Sigma_2^{-1} - (\Sigma+\Sigma_1)^{-1}] y/2 }
\]
which is integrable if and only if the eigenvalues of $\Sigma_2^{-1} - (\Sigma+\Sigma_1)^{-1}$ are positive.
Hence on $\R^d$, a sufficient condition for the existence and uniqueness of a solution to the system \eqref{SchrodingerSystem} is that the eigenvalues of  $\Sigma_2^{-1} - (\Sigma+\Sigma_1)^{-1}$ or $\Sigma_1^{-1} - (\Sigma+\Sigma_2)^{-1}$ are positive.

\subsection{Proof of Theorem I}
Assume (H) and \eqref{cond:thm1} true.\\

The proof introduced by Fortet heavily relies on various monotonicity properties of an iterative scheme. The architecture of the proof is as follows:
\begin{enumerate}[label = Step \arabic*)]
\item The problem is first reduced to proving an equivalent statement;
\item A proper functional space for the iteration scheme is defined;.
\item The iteration scheme is introduced. Its monotonicity properties are established;
\item Two separate cases are identified. In the first case, the iteration scheme converges in a finite number of steps. The existence of a fixed point solution to the problem is then deduced;
\item In the second case, the existence of a fixed point solution to the problem is also proved.
\end{enumerate}

\subsubsection{Step 1: Preliminary Reduction \cite[pp. 86-87]{For}}
Note that system \eqref{SchrodingerSystem} is equivalent to the following system:
\begin{equation}\label{SchrodingerSystem2}
\begin{cases}
(S'1) \;\; {\displaystyle \phi(x) = \frac{\omega_1(x)}{{\displaystyle \int_{\I^2} g(x,y) \frac{\omega_2(y)}{\left[ {\displaystyle \int_{\I^1} g(z,y) \phi(z) dz } \right]} dy}}},  \\
\\
(S'2) \;\; {\displaystyle \psi(y) = \frac{\omega_2(y)}{{\displaystyle \int_{\I^1}g(x,y) \phi(x) dx}}}. 
\end{cases} \tag{S'}
\end{equation}
It suffices to find a solution $\phi$ of (\ref{SchrodingerSystem2}1) to get $\psi$ from (\ref{SchrodingerSystem2}2), and hence solve \eqref{SchrodingerSystem2}.
\\
\noindent
Consider instead the solution of the equation
\begin{equation} \label{eq:h}
h(x) = \int_{\I^2} g(x,y) \frac{\omega_2(y)}{\left[ {\displaystyle \int_{\I^1} g(z,y) \frac{\omega_1(z)}{h(z)} dz  }\right]} dy
\end{equation}
which we shall formally write as
\[
 h = \Omega(h) \tag{1'} \label{eq:h'}
\]
Every solution of \eqref{eq:h} which isn't a.e. zero or infinite yields a solution $\phi$ of (\ref{SchrodingerSystem2}1) by:

\begin{equation} \label{eq:hphi}
\phi(x) = \frac{\omega_1(x)}{h(x)}
\end{equation}
Note that \eqref{eq:hphi} does not define $\phi(x)$ for values of $x \in \I^1$ such that $\omega_1(x) = h(x) = 0$. We shall show, however, that there exists a solution $h$ such that $h(x)>0$ everywhere. Thus, we shall devote our attention to finding a solution $h$ to equation \eqref{eq:h} or, equivalently, to finding a fixed point of the map $\Omega$. The proof relies on an iterative scheme and thus requires introducing a suitable functional space to study the iteration. We  introduce the space of functions of class (C) as:
\begin{definition}  {\textbf{(Step $2$)}} {\textbf{[Function of Class (C)]\cite[p.87]{For}}}\label{def:classC}
$H:\I^1 \rightarrow (\R\cup\infty)$ is a function of class (C) if:
\begin{enumerate}[label = \roman*)]
\item $H$ is measurable;
\item There exists $c>0$ such that for every $x \in \I^1$, we have:
\[
H(x)\ge c;
\]
\item For almost every $x \in \I^1$,
\[
H(x) < +\infty.
\]
\end{enumerate}
\end{definition}
Functions of class (C) are a natural inputs for the map $\Omega$ as the following result (Remark II on p. 89 in \cite{For}) shows.
\begin{remark}\label{remark:classC}
\begin{enumerate}[label = \arabic*.]
\item $H \equiv 1$ is of class (C).
\item If $H_1$ is of class (C), and $H_2$ is measurable, finite a.e., and $H_1 \leq H_2$ everywhere, then $H_2$ is of class (C).
\item If $H_2$ is of class (C), and $H_1$ is measurable, $c<H_1$ everywhere for some $c>0$,and $H_1 \leq H_2$ almost everywhere, then $H_1$ is of class (C)
\item If $H_1$ and $H_2$ are of class (C), then $\max(H_1,H_2)$ and $\min(H_1, H_2)$ are of class (C). \footnote{In this paper, the maximum or minimum of two functions will always be taken pointwise.}
\end{enumerate}
\end{remark}
The following properties are never explicitly stated in \cite{For}.
\begin{proposition}[Properties of $\Omega$]\label{prop:MonotoneOmega}
The map $\Omega$ defined in \eqref{eq:h'} is isotone on functions of class (C), meaning that if $H,H'$ are of class (C) such that 
\[
H \leq H'\quad {\rm a.e.},
\]
then 
\[
\Omega(H) \leq \Omega(H')
\]
everywhere. Moreover, $\forall c>0$ and $H$ of class (C) one has $\Omega(c H) = c \; \Omega(H)$, namely $\Omega$ is positively homogeneous of degree one.
\end{proposition}
\begin{proof}
Suppose $H \leq H'$ a.e. Then, 
\[\frac{\omega_1}{H} \geq \frac{\omega_1}{H'}\quad {\rm a.e.}
\] 
By non-negativity of all the  involved quantities, we get
\[
\int_{\I^2} g(x,y) \frac{\omega_2(y)}{\left[ {\displaystyle \int_{\I^1} g(z,y) \frac{\omega_1(z)}{H(z)} dz  }\right]} dy \leq \int_{\I^2} g(x,y) \frac{\omega_2(y)}{\left[ {\displaystyle \int_{\I^1} g(z,y) \frac{\omega_1(z)}{H'(z)} dz  }\right]} dy
\]
for every $x\in\I^1$. The second property is evident. 
\end{proof}
\subsubsection{Lemma for Functions of Class (C)}
Unfortunately, class (C) is not invariant under map $\Omega$, since the image of a class (C) function might not admit a positive lower bound. Images of class (C) functions under $\Omega$ are however `nearly' of class (C), which is part of the content of his Lemma  \cite[p.89]{For} (notice that we added point $(iv)$ below which  is not in the original statement):
\begin{customlemma}{Lemma}[\cite{For}, p.89]\label{lemma:classC}
Let $H$ be a function of class (C). Define $A = \{ x\in \I^1 | \omega_1(x) > 0\}$.

Let $H' = \Omega(H)$

Then:
\begin{enumerate}[label = \roman*)]
\item $H'$ is measurable;
\item For all compact sets $\mathcal{K} \subseteq \I^1$, there exists a constant $c>0$, depending on $\mathcal{K}$, such that
\[
c<H'(x), \quad \forall x \in \mathcal{K};
\]
\item $H'(x) < +\infty$ for almost every $x \in A$;
\item \begin{equation*} 
\int_{\I^1} \frac{H'(x)}{H(x)} \omega_1(x) dx  = 1;
\end{equation*}
\item If we have moreover $H'(x) \leq H(x)$ or $H'(x) \geq H(x)$ for almost every $x \in A$, then $H'(x) = H(x)$ for almost every $x \in A$.
\end{enumerate}
\end{customlemma}

\begin{proof}
Let $H$ be a function of class (C). In particular, there exists $c>0$ such that $c<H$ everywhere.

Consider two sequences of compact sets $\I^1_1,...,\I^1_n,...$, $\I^2_1,...,\I^2_n,...$ such that:
\[
\begin{cases}
{\displaystyle \I^1_n \subseteq \I^1_{n+1}}, \quad {\displaystyle \I^2_n \subseteq \I^2_{n+1}}, \quad \forall n \in \mathbb{N}^*\\
\I^1_n \uparrow \I^1, \quad \I^2_n \uparrow \I^2, \quad as\;\; n \to +\infty
\end{cases}
\]

Define $\forall y\in \I^2$, $\forall n \in \mathbb{N}^*$ 
\[
G_n(H,y) = \int_{\I^1_n} g(z,y) \frac{\omega_1(z)}{H(z)} dz
\]

First, $G_n(H,\cdot)$ is well defined since $0<c<H$ and $\I^1_n$ is bounded.

Second, for all $y$, there exists $n_y$ such that for all $n\ge n_y$ $G_n(H,\cdot)>0$  from \ref{Hi}-\ref{Hiii} and \ref{Hvii}.

Third, $G_n(H,\cdot)$ is continuous by \ref{Hiv} and the fact that $\I^1_n$ is bounded.

Besides, $G_n(H,y)$ is a non-decreasing sequence in $n$, and from \ref{Hiii},\ref{Hv} we have:
\[
G_n(H,y) \leq \frac{\Sigma}{c} \int_{\I_n^1} \omega_1(z)dz \leq  \frac{\Sigma}{c} 
\]
Which implies that $G_n(H,\cdot)$ is uniformly bounded from above in $n$. Hence by monotone convergence theorem, it admits a pointwise limit
\[
G(H,y) \equiv \int_{\I^1} g(z,y) \frac{\omega_1(z)}{H(z)} dz = \lim_{n \to +\infty} G_n(H,y)
\]
that is a measurable function in $y$, finite everywhere, and positive by monotonicity.

We actually have better than positivity:
\begin{claim}[\cite{For}, p.88]\label{Claim}
For any compact $K \subseteq \I^2$, there exists a constant $\alpha_{K}> 0$ such that 
\[
G(H,y) > \alpha_{K} > 0, \quad \forall y \in K
\]
\end{claim}
\begin{proofclaim}
By monotonicity of the sequence $(G_n(H,y))_n$, it suffices to show this property on some $G_n(H,y)$ for some $n \in \mathbb{N}^*$.

We are thus seeking to prove that for any compact $
K \subseteq \I^2$, there exists some $n \in \mathbb{N}^*$, and a constant $\alpha_{K,n}>0$ such that for any $y \in K$,  
\[
G_n (H,y)  > \alpha_{K,n} >0
\]
We will proceed to a proof by contradiction.

Choose such a $K$. Assume that for all $n,k \in \mathbb{N}^*$, we can find some $y_k \in K$ where
\[
G_n (H,y_k)  < \frac{1}{k}
\]
Choose $n_0$ large enough such that $\frac{\omega_1}{H} > 0$ a.e. on a set $I' \subseteq \I^1_{n_0}$, of positive measure. Such an $n_0$ and $I'$ exist since 
\[
\I^1_n \uparrow \I^1,\quad \int_{\I^1} \omega_1(z) dz = 1,
\]
and $H$, being of class (C), is a.e. finite. 
According to our assumption, for any $k$, there exists $y_{k} \in K$ such that 
\[
G_{n_0} (H,y_{k}) =  \int_{\I^1_{n_0}} g(z,y_{k}) \frac{\omega_1(z)}{H(z)} dz < \frac{1}{k}.
\]
As $k \to +\infty$, $y_{k}$ converges to a limit $y \in K$, up to extracting a subsequence, since $K$ is compact.

Moreover, $H \ge c$ by Definition \ref{def:classC}, and hence
\[
0 \leq g(z,y_{k}) \frac{\omega_1(z)}{H(z)} < \frac{\Sigma}{c} \omega_1(z), \quad \forall k
\]
which is integrable by \ref{Hiii}. By the dominated convergence theorem, one can pass to the limit inside the integral $G_{n_0}(H,y_{k})$ as $k \to +\infty$ and deduce from the continuity of $g$ that:
\[
\int_{\I^1_{n_0}} g(z,y) \frac{\omega_1(z)}{H(z)} dz = 0
\]
By non-negativity of the integrand, for such a $y$, we have:
\[
g(z,y) \frac{\omega_1(z)}{H(z)} = 0, \quad \text{for almost every } z \in \I^1_{n_0}
\]
This is in particular true for almost every $z \in I' \subseteq \I^1_{n_0}$. 

Recall that for almost every $z \in I'$, $\frac{\omega_1(z)}{H(z)} > 0$.

This implies that  
\[
g(z,y) = 0, \quad \text{for almost every } z \in I'
\]
This contradicts \ref{Hvii} since $I'$ has positive measure, and concludes the proof of the claim.
\end{proofclaim}

We can then conclude that $G(H,y)  > \alpha_m > 0$ $\forall m \in \mathbb{N}^*, y \in \I^2_m$ thanks to the monotonicity of the sequence of $G_n (H,y)$.
We can define for $n \in \mathbb{N}^*$ large enough, $x \in \I^1$:
\begin{equation}\label{newH_n'}
H'_{|n}(x) = \int_{\I^2_n} g(x,y) \frac{\omega_2(y)}{{\displaystyle G(H,y) }}dy.\footnote{\linespread{1}\small In Fortet's paper, $H'_{|n}$ is  denoted $H'_n$ \cite[p.88]{For}. Unfortunately, the same notation is later used for  another quantity \cite[p.90]{For}.}
\end{equation}
This integral is well defined and finite since we showed that $G(H,y)>\alpha_n>0$ for $y \in \I^2_n$, is continuous by \ref{Hiv} and non-decreasing in $n$. We can thus set
\[
H'(x) = \lim_{n \to +\infty} H'_{|n}(x)
\]
to be the pointwise limit (potentially infinite) for every $x \in \I^1$.
$H'$ is measurable, positive and bounded from below by a positive constant on any compact $\mathcal{K} \subseteq \I^1$. The proof of the validity of these properties for $H'$ follows the very same pattern as that for $G(H,y)$. This proves i) and ii).
To prove iii),iv) and v), define: 
\[
F(x,y) = g(x,y) \frac{\omega_2(y)}{G(H,y)} \frac{\omega_1(x)}{H(x)}.
\]
$F(x,y)$ is measurable, non-negative, and bounded for $x \in \I^1_q$, $y \in \I^2_p$, for any $p,q \in \mathbb{N}^*$. This because $g$ is bounded from above, $G(H,\cdot)$ and $H$ are bounded from below by positive constants, and $\omega_1, \omega_2$ are continuous on these compact sets.
We then define
\begin{align}
I_{p,q} &= \iint_{\I^1_q \times \I^2_p} F(x,y) dx dy \\
&= \int_{\I^2_p} \omega_2(y) \frac{G_{q}(H,y)}{G(H,y)} dy = \int_{\I^2} \omega_2(y) \1_{\I^2_p}(y) \frac{G_{q}(H,y)}{G(H,y)} dy \label{eq:Ipq1} \\
&= \int_{\I^1_q} \omega_1(x) \frac{H'_{|p}(x)}{H(x)} dx =\int_{\I^1} \omega_1(x) \1_{\I^1_q}(x) \frac{H'_{|p}(x)}{H(x)} dx   \label{eq:Ipq2}
\end{align}
where we used the Fubini-Tonelli theorem to exchange the order of integration, and we denoted by $\1_{\I}$ the indicator function of the set $\I$.
Furthermore, the monotonicity (in the sense of inclusion) of the sets $\I^1_q,\I^2_p$ and monotonicity of the sequences $H'_{|p}(x),G_{q}(H,y)$ implies the monotonicity of the functions $H'_{|p} \1_{\I^1_q} $ and $G_{q}(H,y) \1_{\I^2_p}$, respectively in $p$ and $q$.
One can then use the Beppo-Levi monotone convergence theorem to take limits as $p$ and $q \to \infty$ inside the integrals in \eqref{eq:Ipq1}, \eqref{eq:Ipq2} to infer first from \eqref{eq:Ipq1} that 
\[
\lim_{p \to +\infty} \lim_{q \to +\infty} I_{p,q} = \int_{\I^2} \omega_2(y) \lim_{p \to +\infty} \1_{\I^2_p}(y) dy = 1.
\] 
It then follows from \eqref{eq:Ipq2} that:
\[
\lim_{p \to +\infty} \lim_{q \to +\infty} I_{p,q} = \int_{\I^1} \omega_1(x) \lim_{p \to +\infty} \frac{H'_{|p}(x)}{H(x)} dx  = 1
\]
which gives
\begin{equation} \label{eq:H'H}
\int_{\I^1} \frac{H'(x)}{H(x)} \omega_1(x) dx  = 1
\end{equation}
Recalling that $A = \{ x\in \I^1 | \omega_1(x) > 0\}$, we derive from \eqref{eq:H'H} that $H'$ is finite a.e. on $A$, otherwise the integral in \eqref{eq:H'H} would be infinite. This establishes iii) and iv).
Finally, assume that for almost every $x \in A$ one has either 
\[
H'(x) \leq H(x), \quad or \quad H'(x) \geq H(x)
\]
Then \eqref{eq:H'H} allows us to conclude that $H'=H$ a.e. on $A$, otherwise we would contradict the fact that $\omega_1$ integrates to one.
This establishes v), and completes the proof of the lemma. 
\end{proof}
The following remark appears as Remark I on p. 89 in \cite{For}.
\begin{remark}\label{remark:lemma}
The lemma remains valid if we only assume that $H$ is measurable but only bounded from below by $0$, as long as we can guarantee that the integral $G(H,y)$ remains finite a.e. in $y$. We can even allow $G(H,y)$ to be infinite for values of $y$ where $\omega_2(y) = 0$.
\end{remark}

The above lemma allows us to extract sufficient information on $H' = \Omega(H)$ in order to proceed to the iteration scheme, and prove the first existence result Theorem \ref{thm:ex1}.

\subsubsection{Step 3: Iterative Procedure}
Starting from $H_1 \equiv 1$, one would like to proceed to successive iterations of $\Omega$ by setting $H_{n+1} = \Omega(H_n)$, and show convergence. As illustrated by the \ref{lemma:classC}, if $H$ is of class (C), then $\Omega(H)$ is not necessarily of class (C). Thus, there is no guarantee of obtaining an a.e. finite function if one applies the map $\Omega$ one more time.
Moreover, one has to guarantee the convergence of such an iteration scheme.
Fortet therefore introduces a \textbf{truncation} procedure between two successive iterations of $\Omega$ that takes care of these issues.
The approximation scheme reads \cite[p.90]{For}:
\begin{align*} \label{approxscheme} \tag{AS}
& H_1 \equiv 1, \hspace{2.8cm} H_1' = \Omega(H_1), \quad H_1'' = \min(H_1, H_1') \\
& H_n = \max\left(H_{n-1}'', \frac{1}{n}\right), \quad H_n' = \Omega(H_n), \quad H_n'' = \min(H_1, H_n'), \quad \forall n \geq 2
\end{align*}
The $\max$ step guarantees that $H_n$ always remains in the class (C), and hence we can apply $\Omega$ in the iteration scheme. The vanishing lower bound will lead to a fixed point of $\Omega$ which is not necessarily of class (C).
As for the $\min$ step, it is needed, in particular,  to guarantee the monotonicity of the scheme. 

\begin{innercustomobs}\label{remark:H'1finite}
Note that condition \eqref{cond:thm1}, as well as assumption \emph{(H.v)}, guarantees the (everywhere) finiteness of $H_1' = \Omega(H_1)$, since
\[
H_1'(x) = \Omega(H_1)(x) = \int_{\I^2} g(x,y) \frac{\omega_2(y)}{\left[{\displaystyle \int_{\I^1} g(z,y) \omega_1(z) dz }\right]} dy < \Sigma \int_{\I^2}\frac{\omega_2(y)}{\left[{\displaystyle \int_{\I^1} g(z,y) \omega_1(z) dz }\right]} dy< +\infty. 
\]
\end{innercustomobs}

The following result is stated, but not proven, on \cite[p.90]{For}.

\begin{proposition}[Monotonicity of the scheme (\ref{approxscheme})] \label{lemma:monotonescheme}
For $H_n$, $H_n'$ defined by the scheme {\em (\ref{approxscheme})}, one has $\forall n \in \mathbb{N}^*$:
\[
H_{n+1} \leq H_n, \quad H_{n+1}' \leq H_n'
\]
everywhere.
\end{proposition}

\begin{proof}
By the monotonicity property of $\Omega$ in Proposition~\ref{prop:MonotoneOmega}, it suffices to show that $H_{n+1} \leq H_n$ to deduce that $H_{n+1}' \leq H_n'$, since by definition $H_n' = \Omega(H_n)$, $\forall n \in \mathbb{N}^*$. We prove $H_{n+1} \leq H_n$ by induction.
For $n=1$:
\[
H_1'' = \min(H_1, \Omega(H_1)) =  \begin{cases} 1 \text{ if } \Omega(H_1) \geq 1, \\ \Omega(H_1) \text{ if } \Omega(H_1) \leq 1 .\end{cases}
\]
Thus
\[
H_2 = \max\left(H_1'', \frac{1}{2}\right) =  \begin{cases} 1, \text{ if } \Omega(H_1) \geq 1, \\ \min(\Omega(H_1),\frac{1}{2}), \text{ if } \Omega(H_1) \leq 1, \end{cases} \leq 1 = H_1,
\]
which proves the initialization step of the induction.
Let us now assume that the property is true for some $n \in \mathbb{N}^*$, namely we have 
\[
H_{n+1} \leq H_n
\]
pointwise.
Then, by the monotonicity of $\Omega$ (Proposition \ref{prop:MonotoneOmega}), we have that $H_{n+1}' \leq H_{n}'$, and thus 
\[
H_{n+1}'' = \min(H_{n+1}', H_1) \leq \min(H_{n}', H_1) = H_n''
\]
Since we also have $\frac{1}{n+2} < \frac{1}{n+1}$, we can infer that
\[
H_{n+2} = \max\left( H_{n+1}'',\frac{1}{n+2}\right) \leq  \max\left( H_{n}'',\frac{1}{n+1}\right) = H_{n+1}
\]
which concludes the proof by induction.
\end{proof}
\begin{innercustomobs}\label{remark:finite}
Since $H_1\equiv 1$, each $H_n$ is finite everywhere. In addition, \ref{remark:H'1finite} and Proposition \ref{lemma:monotonescheme} also show that each $H_n'$ is finite everywhere.
\end{innercustomobs}
\noindent
The monotonicity of Proposition~\ref{lemma:monotonescheme} will be crucial to establishing existence of a fixed point for \eqref{eq:h'}. When iterating (\ref{approxscheme}), we distinguish two separate cases which lead to different fixed points:

\subsubsection{First Case {\rm \cite[Section 2, p. 86]{For}}}
In this case, we assume that, as we iterate following the approximation scheme (\ref{approxscheme}), there exists some $n_0 \in \mathbb{N}^*$ such that a.e. one has
\begin{equation} \label{eq:case1}
H_{n_0}' \leq H_1.
\end{equation}
We shall show, using the \ref{lemma:classC}, that $\Omega(H_{n_0}')$ is a solution to equation \eqref{eq:h} (and that $H_{n_0}'$ is `nearly' a solution). We first need to show that $\Omega(H_{n_0}')$ is well defined. This will be accomplished by approximating $H_{n_0}'$ as shown below.
First of all, notice that \eqref{eq:case1} together with the  definition of $H_{n_0}''$ in scheme (\ref{approxscheme}) yields
\begin{equation}\label{eq:H''H'}
H_{n_0}'' = H_{n_0}'.
\end{equation}
Let us define
\begin{equation} \label{eq:Kp}
K_p = \max \left( H_{n_0}', \frac{1}{p} \right), \quad p \in \mathbb{N}^*.
\end{equation}
Although $H_{n_0}'$ may not be of class (C), it follows from the \ref{lemma:classC} that $K_p$ is of class (C) since we have the uniform lower bound $\frac{1}{p}>0$. Furthermore, as $p \to +\infty$, $K_p \to H_{n_0}'$ pointwise.
Set 
\begin{equation*}
K_p' = \Omega(K_p).
\end{equation*}
Note that $K_{p+1} \leq K_p$. By Proposition \ref{prop:MonotoneOmega}, the sequence of $K_p' = \Omega(K_p)$ is also decreasing in $p$. 
By the non-negativity of $K_p'$,  the sequence $\{K_p'\}$ admits a pointwise limit $K'$ which is measurable and non-negative:
\begin{equation} \label{eq:limK'}
K' = \lim_{p \to +\infty} K_p' = \lim_{p \to +\infty} \Omega(K_p).
\end{equation}
Recalling that $\Omega$ was defined as an integral operator, we can then use Beppo-Levi monotone convergence theorem to get from the monotonicity of the sequence of $K_p$ that
\[
\lim_{p \to +\infty} \Omega(K_p) = \Omega \left(\lim_{p \to +\infty} K_p \right) = \Omega(H_{n_0}').
\]
Putting this together with \eqref{eq:limK'}, we finally get
\begin{equation}\label{eq:K'}
K' = \Omega(H_{n_0}').
\end{equation}
To show that $K'$ is a solution of (\ref{eq:h'}), we first need the following result whose statement and sketch of the proof can be found on \cite[p.91]{For}.
\begin{proposition} \label{prop:claim1}
\[
\int_{\I^1} \frac{K'(x)}{H_{n_0}'(x)} \omega_1(x) dx = 1
\]
\end{proposition}
\begin{proof}
From the scheme (\ref{approxscheme}), we know that
\[
H_{n_0} \geq \frac{1}{n_0} = \frac{1}{n_0} H_1
\]
Using both properties of Proposition \ref{prop:MonotoneOmega}, we get
\[
H_{n_0}' \geq \frac{H_1'}{n_0}.
\]
By the definition of $K_p$ \eqref{eq:Kp}, we now get:
\begin{equation} \label{eq:Kps}
K_p \geq H_{n_0}' \geq \frac{H_1'}{n_0}.
\end{equation}
Furthermore, since we assumed that $H_{n_0}' \leq 1$ a.e., one also has by the definition of $K_p$ \eqref{eq:Kp} that $K_p \leq H_1 = 1$ a.e.. This implies, by Proposition \ref{prop:MonotoneOmega} that $K_p' \leq H_1'$ everywhere.
Plugging the latter inequality in \eqref{eq:Kps} yields that $\forall p \in \mathbb{N}^*$,
\begin{equation} \label{eq:Kp'Kp}
\frac{K_p'}{K_p} \leq n_0.
\end{equation}
This implies that, taking the limit for $p \to +\infty$, we also have 
\[\frac{K'}{H_{n_0}'} \leq n_0.\]
Since $K_p$ is of class (C), \ref{lemma:classC} iv) yields 
\[
\int_{\I^1} \frac{K_p'(x)}{K_p(x)} \omega_1(x) dx = 1
\]
By \eqref{eq:Kp'Kp}, the integrand is uniformly bounded in $p$. By \ref{Hiii}, the measure is finite. We conclude by the bounded convergence theorem that
\begin{equation*}
\int_{\I^1} \frac{K'(x)}{H_{n_0}'(x)} \omega_1(x) dx = 1
\end{equation*}
which concludes the proof. 
\end{proof}

Lastly, we shall also need the following result whose statement and sketch of the  proof can also be found on \cite[p.91]{For}.
\begin{proposition} \label{prop:claim3}
We have
\[
K' \leq H_{n_0}'
\]
everywhere on $\I^1$.
\end{proposition}
\begin{proof}
First of all, notice that for $p > n_0+1$, one has from the scheme (\ref{approxscheme}), from \eqref{eq:H''H'} and from the definition of \eqref{eq:Kp} $K_p$:
\[
H_{n_0+1} = \max\left( H_{n_0}'', \frac{1}{n_0+1} \right) = \max\left( H_{n_0}', \frac{1}{n_0+1} \right) = K_{n_0+1}.
\]
By monotonicity of the sequence of $K_p$, we also have that, for $p > n_0+1$, $K_{n_0+1} \geq K_p$ everywhere. This together with the above equality then gives for $p > n_0+1$:
\[
H_{n_0+1} \geq K_p.
\]
Applying $\Omega$ to both sides of the above inequality and using again Proposition \ref{prop:MonotoneOmega}, we get
\[
K_p' \leq H_{n_0+1}'
\]
Since $K_p' \geq K'$ and $H_{n_0+1}' \leq H_{n_0}'$ (Proposition \ref{lemma:monotonescheme}), we finally obtain 
\[
K' \leq H_{n_0}'.
\]
\end{proof}
We now employ Propositions \ref{prop:claim1} and \ref{prop:claim3} to complete the first case: On $A = \{ x \in \I^1 | \omega_1(x) > 0\}$, we must have a.e.
\begin{equation}\label{eq:fixed}
K' = H_{n_0}'.
\end{equation}
Recalling that $K' =\Omega(H_{n_0}')$ (see(\ref{eq:K'})), we conclude from \eqref{eq:fixed} that \[\Omega(H_{n_0}') = H_{n_0}', \quad {\rm a.e. \;on}\; A.\] 
We proceed to show that actually this equality holds on all of $\I^1$.
Indeed, by (\ref{eq:fixed}), for every $y \in \I^2$:
\begin{align*}
G(K',y) &= \int_{\I^1} g(z,y) \frac{\omega_1(z)}{K'(z)} dz = \int_{A} g(z,y) \frac{\omega_1(z)}{K'(z)} dz \\
&= \int_{A} g(z,y) \frac{\omega_1(z)}{H_{n_0}'(z)} dz = \int_{\I^1} g(z,y) \frac{\omega_1(z)}{H_{n_0}'(z)} dz \\
&= G(H_{n_0}',y).
\end{align*}
It follows in view of (\ref{eq:K'}), that for every $x \in \I^1$:
\[
\Omega(K')(x) = \int_{\I^2} g(x,y) \frac{\omega_2(y)}{G(K',y)} dy = \int_{\I^2} g(x,y) \frac{\omega_2(y)}{G(H_{n_0}',y)} dy = \Omega(H_{n_0}')(x) = K'(x). 
\]
Thus, $K'$ is a fixed point of the map $\Omega$. This concludes the proof of the first case. 

The following bounds for $K'$ are merely stated on \cite[p.91]{For}.
\begin{proposition} For every $x \in \I^1$
\begin{equation}\label{eq:0K'1}
0 < K'(x) \leq 1
\end{equation}
\end{proposition}
\begin{proof}
By assumption (\ref{eq:case1}), $H_{n_0}' \leq 1$ a.e which implies $H_{n_0}'' = H_{n_0}'$ (\ref{eq:H''H'}). 
Hence
\[
H_{n_0+1} = \max\left(H_{n_0}'', \frac{1}{n_0+1}\right) = \max\left(H_{n_0}', \frac{1}{n_0+1}\right) = K_{n_0+1}
\]
by definition  \eqref{eq:Kp} of $K_p$. Applying the map $\Omega$ and using Proposition \ref{prop:MonotoneOmega}, we get
\[
H_{n_0}' =\Omega(H_{n_0}) \geq \Omega(H_{n_0+1}) = \Omega(K_{n_0+1}) = K_{n_0+1}'
\]
Since the sequence of $K'_{p}$ monotonically decreases to $K'$, we then conclude that
\[
1 \geq H_{n_0}' \geq  K_{n_0+1}' \geq K'.
\]
Thus, $K' \leq 1$ everywhere. To prove $K'>0$, recall that by (\ref{eq:K'}) 
\[
K'(x) = \Omega(H_{n_0}')(x) = \int_{\I^2} g(x,y)\frac{\omega_2(y)}{\left[ {\displaystyle \int_{\I^1} g(z,y) \frac{\omega_1(z)}{H_{n_0}'(z)} dz} \right]}dy
\]
$H_{n_0}'$ is not necessarily of class (C). In particular, we do not have an \emph{a priori} positive lower bound. Thus we cannot  apply the \ref{lemma:classC} to prove the statement as we cannot \emph{a priori} guarantee that 
\[
 \int_{\I^1} g(z,y) \frac{\omega_1(z)}{H_{n_0}'(z)} dz < +\infty, \quad \text{ a.e. in } y.
\]
Notice instead that since $H_{n_0}' = H_{n_0}''$, we get from the scheme (\ref{approxscheme}):
\[
H_{n_0+1} \geq H_{n_0}'' = H_{n_0}'
\] 
By Proposition \ref{lemma:monotonescheme}, $H_{n_0} \geq H_{n_0+1}$ and thus $H_{n_0} \geq H_{n_0}'$ everywhere. Since $H_{n_0}$ is of class (C), we have by the \ref{lemma:classC} v) that $H_{n_0} = H_{n_0}'$ a.e. on $A=\{ x\in \I^1 | \omega_1(x) > 0\}$.
In particular, there exists a constant $c>0$ such that for a.e. $x \in A$, $H_{n_0}'(x) \geq c$. It follows that
\begin{align*}
\int_{\I^1} g(z,y) \frac{\omega_1(z)}{H_{n_0}'(z)} dz &= \int_{ A} g(z,y) \frac{\omega_1(z)}{H_{n_0}'(z)} dz + \int_{\I^1 \backslash A} g(z,y) \frac{\omega_1(z)}{H_{n_0}'(z)} dz \\
&= \int_{A} g(z,y) \frac{\omega_1(z)}{H_{n_0}'(z)} dz, \quad \text{ since } \omega_1 = 0 \text{ on } \I^1 \backslash A \text{ and } H_{n_0}'(z)>0\\
&\leq \frac{1}{c} \int_{A} g(z,y) \omega_1(z) dz, \quad \text{ since } H_{n_0}' \geq c \text{ on } A \\
& \leq \frac{\Sigma}{c} , \quad \text{ from \ref{Hiii},\ref{Hv}} 
\end{align*}
We conclude that,  for all $x \in \I^1$, we have 
\begin{align*}
K'(x) &= \int_{\I^2} g(x,y)\frac{\omega_2(y)}{\left[ {\displaystyle \int_{\I^1} g(z,y) \frac{\omega_1(z)}{H_{n_0}'(z)} dz} \right]}dy \\
&\geq \frac{c}{\Sigma} \int_{\I^2} g(x,y) \omega_2(y) dy > 0,
\end{align*}
where the last inequality follows from \ref{Hi},\ref{Hiii},\ref{Hvi}. 
\end{proof}
\subsubsection{Second Case  {\rm \cite[Section 2, p. 92]{For}}}
Contrary to the first case, assume now that $\forall n \in \mathbb{N}^*$, there exists a positive measure set $J_n$ on which $H_n'> H_1$.
Define by $H$ and $H'$ the respective limits of the sequences $H_n$ and $H_n'$. By Proposition \ref{lemma:monotonescheme}, nonnegativity of the sequences and  \ref{remark:finite}, these limits exist, are measurable and finite. We shall show that $H'$ is a fixed point of the map $\Omega$.
First notice that the sequence of $J_n$'s is monotonically decreasing:
\begin{proposition}[Monotonicity of $J_n$]\label{prop:Jn}
We have $\forall n \in \mathbb{N}^*$:
\[
J_{n+1} \subseteq J_n
\]
\end{proposition}
\begin{proof}
Let $n \in \mathbb{N}^*, x \in J_{n+1}$. Then $H_{n+1}'(x)> H_1(x)$. By Proposition \ref{lemma:monotonescheme}, $H_{n}'(x) \geq H_{n+1}'(x) > H_1(x)$. Thus, $x \in J_n$. 
\end{proof}
We can then define 
\[\tilde{J} = {\displaystyle \lim_{n \to +\infty} J_n} ={\displaystyle \bigcap_{n \in \mathbb{N}^*} J_n},\]
and
\begin{equation}\label{defJ}
J = \left\lbrace x \in \tilde{J} \:|\: H'(x)>1 \right\rbrace.
\end{equation}
One has moreover \footnote{\linespread{1}\small Fortet seems to imply by this proposition that $H$ and $H'$ cannot vanish at a point without vanishing everywhere. Although this is true for $H'$, see Proposition \ref{prop:HH'0} below, it does not imply the same property for $H$.} the following inequality which is stated on \cite[p.92]{For}.
\begin{proposition}\label{prop:H'H}
\[
H \leq H'
\]
everywhere.
\end{proposition}
\begin{proof}
By the scheme (\ref{approxscheme}) and Proposition \ref{lemma:monotonescheme}, the nonnegative sequence of $H_n''$ is also decreasing. Hence, it admits a limit $H''$.
By definition, $H_n = \max \left( H_{n-1}'',\frac{1}{n} \right)$.  Thus, the limits must be equal $H= H''$.
Since $H_n'' = \min(H_1, H_n') \leq H_n'$, we get, passing to the limit, that $H \leq H'$. \qed
\end{proof}
The following result shows that  $H'$ cannot vanish, otherwise we would fall back in the first case\footnote{\linespread{1}\small The statement can be found on \cite[p.92]{For}. The proof there provided, however, appears to be incorrect as it does not make use of hypothesis  \eqref{cond:thm1} confusing $H_n'$ of the iteration (\ref{approxscheme}) with $H'_{|n}$ (also denoted by $H'_n$ by Fortet) defined in (\ref{newH_n'}).}.

\begin{proposition} \label{prop:HH'0}
Assume that there exists some $x_0 \in \I^1$ such that $H'(x_0) = 0$.
Then the sequence $H_n'$ converges uniformly to $0$ on $\I^1$. In particular, $H' \equiv 0$.
\end{proposition}

\begin{proof}
Assume that there exists some $x_0 \in \I^1$ such that $H'(x_0) = 0$. By definition of $H'$, this implies that the sequence $H_n'(x_0)$ converges to 0, i.e.:
\[
H_n'(x_0) = \int_{\I^2} g(x_0,y) \frac{\omega_2(y)}{G(H_n,y)} dy \to 0, \quad \text{ as } n \to +\infty.
\]
This implies that 
\begin{equation} \label{claim:measure}
\text{The measure } \frac{\omega_2(y)}{G(H_n,y)}dy  \text{ converges strongly (in total variation norm) to } 0 \text{ on } \I^2.
\end{equation} 
The proof of the above statement can be found in the Appendix. Now pick any $x \in \I^1$. We know from \ref{remark:finite} that $H_n'(x)$ is finite. In particular, approximating $\I^2 \supseteq ... \supseteq \I^2_q \supseteq ... \supset \I^2_1$ by compact sets $\I^2_q$, one can write for $q \in \mathbb{N}^*$:
\begin{equation}\label{eq:intH'n}
H_n'(x) = \int_{\I^2_q} g(x,y) \frac{\omega_2(y)}{G(H_n,y)} dy + \int_{\I^2 \backslash \I^2_q} g(x,y) \frac{\omega_2(y)}{G(H_n,y)} dy.
\end{equation}
By boundedness of $\I^2_q$  and \eqref{claim:measure}, the first integral
\[
\int_{\I^2_q} g(x,y) \frac{\omega_2(y)}{G(H_n,y)} dy \leq \Sigma \int_{\I^2_q} \frac{\omega_2(y)}{G(H_n,y)} dy
\]
converges uniformly in $x$ to $0$ as $n \to +\infty$.
As for the second integral, notice that $H_n \leq H_1$ from Proposition \ref{lemma:monotonescheme}. Hence
\[
\int_{\I^2 \backslash \I^2_q} g(x,y) \frac{\omega_2(y)}{G(H_n,y)} dy \leq \int_{\I^2 \backslash \I^2_q} g(x,y) \frac{\omega_2(y)}{G(H_1,y)} dy \leq \Sigma \int_{\I^2 \backslash \I^2_q} \frac{\omega_2(y)}{G(H_1,y)} dy
\]
which can be made, uniformly in $x$, arbitrarily small  when $q \to +\infty$, by absolute continuity of the measure $ \frac{\omega_2(y)}{G(H_1,y)}$ with respects to the Lebesgue measure, thanks to condition \eqref{cond:thm1}.
We therefore conclude the uniform convergence of the sequence of $H_n'$ to $0$.
\end{proof}

It follows from Proposition \ref{prop:HH'0} that if $H'$ vanishes at one point, $H_n'$ converges uniformly to $0$. In that case, for $n$ large enough, we would have for every $x$, $H'(x) \leq 1 = H_1$. We would namely be in the first case. We can then conclude that, in this second case, we necessarily have $H' > 0$ everywhere.
To prove that $H'$ satisfies (\ref{eq:h'}), we shall show  that, although we do not have $H_{n}' \leq 1$ a.e. for some $n$, this holds for the limit $H' $.  The rest of the proof will then be similar to the first case provided we can show that the set $J$ has zero measure. This is stated, followed by a very sketchy proof by contradiction, on \cite[p.93]{For}.
\begin{proposition}\label{prop:Jmeas0}
The set $J$ defined in (\ref{defJ}) has measure $0$.
\end{proposition}
\begin{proof}
Assume that it is not the case. Then one has for $x \in J$, $H'(x)>H_1(x)=1$.
The scheme (\ref{approxscheme}) thus yields $H''(x) = \min(H_1(x),H'(x)) = 1$, and hence 
\begin{equation}\label{eq:HJ}
H(x) = \max(H''(x),0) = 1 < H'(x),\quad \forall x \in J,
\end{equation} 
as well as 
\begin{equation}\label{eq:HJ2}
H(x) = \max(H''(x),0) = 1 \leq H'(x),\quad \forall x \in \tilde{J}.
\end{equation}

Similarly, for $x \in \I^1 \backslash J$, one has from the approximation scheme $H''(x) = \min(H_1(x), H'(x)) = H'(x)$, and hence
\begin{equation}\label{eq:HIJ}
H(x) = \max(H''(x),0) = H'(x), \quad \forall x \in \I^1 \backslash J
\end{equation} 

From \eqref{eq:HJ}, \eqref{eq:HIJ} and the fact that $J$ has positive measure, it follows  that:
\begin{equation}\label{eq:intH'H}
\int_{\I^1} \frac{H'(x)}{H(x)} \omega_1(x) dx = \int_{J} \frac{H'(x)}{H(x)} \omega_1(x) dx + \int_{\I^1\backslash J} \frac{H'(x)}{H(x)} \omega_1(x) dx > \int_{\I^1} \omega_1(x) dx = 1.
\end{equation}
Recall now that for $n \in \mathbb{N}^*$, $H_n$ is of class (C), and hence, by \ref{lemma:classC} iv) we have that:
\begin{equation}\label{eq:intHn'Hn}
\int_{\I^1} \frac{H_n'(x)}{H_n(x)} \omega_1(x) dx = 1.
\end{equation}
The strategy consists in passing to the limit in the above equation and derive a contradiction with \eqref{eq:intH'H}. However, passing to the limit is delicate, hence we consider the following decomposition: 
\begin{align}
1 = \int_{\I^1} \frac{H_n'(x)}{H_n(x)} \omega_1(x) dx &=\int_{\I^1\backslash J_{n-1}} \frac{H_n'(x)}{H_n(x)} \omega_1(x) dx + \int_{J_{n-1}} \frac{H_n'(x)}{H_n(x)} \omega_1(x) dx \\
 &= \int_{\I^1\backslash J_{n-1}} \frac{H_n'(x)}{H_n(x)} \omega_1(x) dx + \int_{J_{n-1}} H_n'(x) \omega_1(x) dx \label{eq:intHn'Hn2}
\end{align}
since again by the scheme, one has for $x \in J_{n-1}$: $H_{n-1}''(x) = 1$ and, consequently, $H_n(x)=1$.
Now notice that by monotonicity of $J_n$'s (Proposition \ref{prop:Jn}) and of $H_n'$ (Proposition \ref{lemma:monotonescheme}), one has:
\begin{equation}\label{eq:intJn}
\int_{J_{n-1}} H_n'(x) \omega_1(x) dx \geq \int_{\tilde{J}} H_n'(x) \omega_1(x) dx \geq \int_{\tilde{J}} H'(x) \omega_1(x) dx = \int_{\tilde{J}} \frac{H'(x)}{H(x)} \omega_1(x) dx
\end{equation}
where the last equality holds because of \eqref{eq:HJ2}.\\
Let us now focus on the second integral: 
\[
\int_{\I^1\backslash J_{n-1}} \frac{H_n'(x)}{H_n(x)} \omega_1(x) dx  = \int_{\I^1} \1_{\I^1\backslash J_{n-1}}(x) \frac{H_n'(x)}{H_n(x)} \omega_1(x) dx 
\]
Notice that for $x \in \I^1\backslash J_{n-1} $, we have $H_{n-1}'(x) \leq 1$.
We then get from the scheme (\ref{approxscheme}) that $H_{n-1}''(x) = H_{n-1}' (x)$. It follows that
\[
\text{ either } H_n(x) = H_{n-1}'(x), \quad \text{ or } \quad H_n(x) = \frac{1}{n} \text{ in the case } H_{n-1}'(x) \leq \frac{1}{n}
\]
In any case one has
\[
H_n(x) \geq H_{n-1}'(x).
\]
By Proposition \ref{lemma:monotonescheme},  $H_{n-1}'(x) \geq H_{n}'(x)$. We get that
\[
\frac{H_{n}'(x)}{H_{n}(x)} \leq 1, \quad \forall x \in \I^1\backslash J_{n-1}.
\]
The bounded convergence theorem allows us to conclude that
\begin{equation}\label{eq:limintH'H}
\lim_{n \to +\infty} \int_{\I^1\backslash J_{n-1}} \frac{H_n'(x)}{H_n(x)} \omega_1(x) dx = \int_{\I^1} \lim_{n \to +\infty} \1_{\I^1\backslash J_{n-1}}(x) \frac{H_n'(x)}{H_n(x)} \omega_1(x) dx = \int_{\I^1\backslash \tilde{J}} \frac{H'(x)}{H(x)} \omega_1(x) dx.
\end{equation}
Using \eqref{eq:intJn} and \eqref{eq:limintH'H} into \eqref{eq:intHn'Hn2}, one gets when passing to the limit in $n$ that
\[
1 \geq \int_{\tilde{J}} \frac{H'(x)}{H(x)} \omega_1(x) dx + \int_{\I^1\backslash \tilde{J}} \frac{H'(x)}{H(x)} \omega_1(x) dx = \int_{\I^1} \frac{H'(x)}{H(x)} \omega_1(x) dx
\]
which contradicts \eqref{eq:intH'H}.
\end{proof}

Now that we know that $J$ is of measure $0$, we are ready to  show that $H'$ is indeed a solution of \eqref{eq:h'}. Since we have $H_n' = \Omega(H_n)$, Beppo-Levi's monotonce convergence theorem implies that $H' =\Omega(H)$. Since $J$ is of measure $0$, we have that $H' \leq H_1 =1 $ a.e.. By the definition $H_n'' = \min(H_1, H_n')$, passing to the limit we then get $H'' = H'$ a.e.. This, together with $H_n = \max\left(H_{n-1}'',\frac{1}{n}\right)$, also gives $H = H'' = H'$ a.e. . We conclude that everywhere:
\[
H' = \Omega(H) = \Omega(H'), 
\] 
which proves that $H'$ is solution to \eqref{eq:h'}.

\subsubsection{Conclusion}
To summarize, in both cases we found a measurable solution $h$ of \eqref{eq:h'} ($h = K'$ in the first case, $h = H'$ in the second case) such that we have everywhere
\[
0 < h \leq 1.
\] 
Moreover, we have continuity of the solution. This is stated with a sketch of the proof on \cite[p.95]{For}.
\begin{proposition}\label{prop:hcont}
$h$ is continuous on $\I^1$.
\end{proposition}
\begin{proof}
Recall the definition
\[
G(H,y) \equiv \int_{\I^1} g(z,y) \frac{\omega_1(z)}{H(z)} dz
\]

Since $h \leq 1= H_1$ everywhere in $x$, we get $\frac{\omega_2(y)}{G(h,y)} \leq \frac{\omega_2(y)}{G(H_1,y)}$ everywhere in $y$.
Then, for $x_1,x_2 \in \I^1$, 
\[
|h(x_1) - h(x_2)| \leq \int_{\I^2} |g(x_1,y)-g(x_2,y)| \frac{\omega_2(y)}{G(h,y)} dy\leq \int_{\I^2} |g(x_1,y)-g(x_2,y)| \frac{\omega_2(y)}{G(H_1,y)} dy
\] 
From \ref{Hv}, $|g(x_1,y)-g(x_2,y)| \frac{\omega_2(y)}{G(H_1,y)} \leq 2 \Sigma \frac{\omega_2(y)}{G(H_1,y)}$, which is integrable by \eqref{cond:thm1}.
Thus one can use the dominated convergence theorem to deduce that
\[
\lim_{x_2 \to x_1}|h(x_1) - h(x_2)| \leq \int_{\I^2} \lim_{x_2 \to x_1} |g(x_1,y)-g(x_2,y)| \frac{\omega_2(y)}{G(H_1,y)} dy = 0
\]
from the continuity of $g$.

\end{proof}
We now reformulate the existence results and the properties of $h$ in terms of the original variables $(\phi,\psi)$.
Since $0<h\leq 1$ everywhere, equation \eqref{eq:hphi} defines a proper measurable function $\phi$ on $\I^1$, non-negative and only vanishing for values $x$ where $\omega_1(x) = 0$, which is moreover continuous from \ref{Hviii}, Proposition \ref{prop:hcont} and the fact that $h>0$.
This proves Theorem \ref{thm:ex1}.ii),iii) and the existence of $\phi$ solution of (\ref{SchrodingerSystem2}1).
Given such a $\phi$, one can define a measurable solution $\psi$ from (\ref{SchrodingerSystem2}2). By property $\phi \geq 0$, \ref{Hi},\ref{Hii}, we have that $\psi \geq 0$, which proves Theorem \ref{thm:ex1}.i),iv).
It remains to establish Theorem \ref{thm:ex1}.v).
Let $A' = \{ y \in \I^2 | \omega_2(y) > 0 \}$, and $A'' = \{ y \in A' | \psi(y) = 0\} \subseteq A'$. The goal is to show that $A''$ has measure $0$.
To this end, we compute:
\[
\int_{\I^2} g(x,y) \psi(y) dy = \int_{A' \backslash A''} g(x,y) \psi(y) dy
\]
since by (\ref{SchrodingerSystem2}2), $\psi = 0$ outside of $A'$, and by definition of $A''$, $\psi = 0$ on $A''$. We can then multiply the above equation by $\phi(x)$ and integrate over $\I^1$. Since all functions involved are non-negative and measurable, one can decide the order of integration by Fubini-Tonnelli. On the one hand, we have
\begin{align*}
\int_{\I^1} \phi(x) \left[\int_{A' \backslash A''} g(x,y) \psi(y) dy \right] dx&= \int_{\I^1} \phi(x)\left[\int_{\I^2} g(x,y) \psi(y) dy \right] dx\\
&= \int_{\I^2} \psi(y) \left[\int_{\I^1} \phi(x) g(x,y) dx\right] dy \\
&= \int_{\I^2} \omega_2(y) dy,\quad \text{ from (S'2)} \\
&= 1
\end{align*}
On the other hand, we get:
\begin{align*}
\int_{\I^1} \phi(x)\left[\int_{A' \backslash A''} g(x,y) \psi(y) dy \right] dx&= \int_{A' \backslash A''} \psi(y) \left[\int_{\I^1} \phi(x) g(x,y) dy \right] dx \\
&= \int_{A' \backslash A''} \omega_2(y) dy, \quad \text{ from (S'2)} 
\end{align*}
We deduce that 
\[
\int_{A' \backslash A''} \omega_2(y) dy = \int_{A'} \omega_2(y) dy
\]
which is only possible if $A''$ has measure $0$, since $\omega_2>0$ on $A'' \subset A'$.
This concludes the proof of Theorem \ref{thm:ex1}. \qed
\section{Second Existence Theorem and Uniqueness Theorem}\label{Existence II}
In \cite[Section 3, pp. 97-102]{For}, Fortet proceeds to derive an existence theorem for System (\ref{SchrodingerSystem}) still under hypotheses \ref{Hi}-\ref{Hviii} but without assuming the integrability condition (\ref{cond:thm1}). The latter condition is replaced by the assumption that the kernel function $g(x,y)$ be of class (B) \cite[p. 97]{For}. The latter property appears in general hard to check. We have therefore decided to present only a special case of the second existence theorem where this property can be readily verified. 
\begin{theoremR} {\em \cite[p. 101]{For}}\label{thm:ex2}
Suppose $\I^1=\I^2=\R$ and that $g(x,y)=U(x-y)$ only depends on the difference $t=x-y$. Assume, moreover, that one of the following conditions is met: $1)$ there exist $T_1\le T_2$ such that for $t\le T_1$ $U(t)$ is non decreasing and for $t\ge T_2$  it is non increasing  $2)$ there exist $T_1\le T_2$ such that for $t\le T_1$ $U(t)$  is non increasing and  for $t\ge T_2$ it is non decreasing. Assume, finally, {\em \ref{Hi}-\ref{Hviii}}. Then system (\ref{SchrodingerSystem}) admits a solution $(\varphi(x),\psi(y))$. The function $\varphi(x)$ is zero for the values $x$ for which $\omega_1(x)$ is zero. On the complement, $\varphi$ is strictly positive and continuous. The non negative function $\phi(y)$ is measurable and equal to zero, up to a zero measure set, only for the values $y$ where $\omega_2(y)=0$.
\end{theoremR}
\begin{innercustomobs} Notice that this theorem applies to the important case where $g(x,y)=p(0,x,1,y)$ the heat kernel {\em (\ref{transitiondensity})} and arbitrary continuous densities $\omega_1(x)$ and $\omega_2(y)$ with support equal to the real line.
\end{innercustomobs}
\begin{proof}[Sketch of the Proof of Theorem \ref{thm:ex2}, pp.98-101]
\begin{enumerate}
\item A continuous, positive function $\rho$ is introduced which satisfies, in particular, the following property:
\[
\int_{\I^1} \frac{\omega_1(x) \rho(x)}{\left[ {\displaystyle \int_{\I^2} g(x,z) \omega_2(z) dz} \right]} dx < +\infty.
\]
\item By Theorem \ref{thm:ex1} and by construction of $\rho$, the system
\[
\begin{cases}
{\displaystyle \bar{\phi}(x) \int_{\I^2} g(x,y) \bar{\psi}(y) dy} = \omega_1(x) \rho(x), \\
\\
{\displaystyle \bar{\psi}(y) \int_{\I^1} g(x,y) \bar{\phi}(x) dx} = \omega_2(y),
\end{cases}
\]
admits a solution $(\bar{\phi},\bar{\psi})$.
\item The same techniques as in the \ref{lemma:classC} and the proof of Theorem \ref{thm:ex1} permit to show that there exists a fixed point for the operator $\bar{\Omega}$ defined on functions of class (C) by:
\[
\bar{\Omega} (H) (x)= \frac{\bar{\phi}(x)}{\omega_1(x)} \int_{\I^2} g(x,y) \frac{\omega_2(y)}{ \left[ {\displaystyle \int_{A} g(z,y) \frac{\bar{\phi}(z)}{H(z)} dz} \right] }dy.
\]
The fixed point $\overline{h}$, which is not necessarily of class (C), enjoys  properties similar to the fixed point of $\Omega$ defined in \eqref{eq:h'}.
\item Set 
\[
\begin{cases}
\phi(x) = {\displaystyle \frac{\bar{\phi}(x)}{\overline{h}(x)}}, \quad x \in A, \\
\phi(x) = 0, \quad x \in \I^1 \backslash A.
\end{cases}
\]
Then $\phi$ is a solution of (\ref{SchrodingerSystem2}1). The other function $\psi$ can then be recovered from (\ref{SchrodingerSystem2}2).
\end{enumerate}
The assumptions of Theorem \ref{thm:ex2} are used to show that the various integrals in this proof are well defined.
\end{proof}

Fortet defines as a nonnegative (positive in French) solution of (\ref{SchrodingerSystem}) to be a pair of nonnegative functions $(\varphi(x),\psi(y))$ satisfying (\ref{SchrodingerSystem}) and the following properties: They are a.e. finite, and different from zero (up to a zero measure set) for the values  where $\omega_1\neq 0$ and $\omega_2\neq 0$, respectively. Moreover, under hypotheses \ref{Hi}-\ref{Hviii}, $\varphi(x)$ is zero at the same time as $\omega_1$ and $\psi$ is zero at the same time as $\omega_2$. The proof of the following uniqueness theorem \cite[pp.102-104]{For} has been slightly reformulated and completed.
\begin{theoremR} {\em \cite[p. 104]{For}}\label{thm:uniq}
Assume {\em \ref{Hi}-\ref{Hviii}}. Let $(\varphi_1,\psi_1)$ and $(\varphi_2,\psi_2)$ be two nonnegative and measurable solutions of system (\ref{SchrodingerSystem}). Then, there exists a positive constant $c$ such that 
\begin{equation}\label{ray}\frac{\varphi_1(x)}{\varphi_2(x)}\equiv c\equiv \frac{\psi_2(y)}{\psi_1(y)}.
\end{equation}
\end{theoremR}
\begin{proof} 
Let $(\phi_1,\psi_1)$, $(\phi_2,\psi_2)$ be two solutions of System \eqref{SchrodingerSystem}. According to Theorem \ref{thm:ex1} or \ref{thm:ex2}, $\phi_1$ and $\phi_2$ are positive and finite on the support $A_1$ of $\omega_1$. Hence, there exists a value of $x_0\in A_1$ such that 
\[
0<\phi_1(x_0)<+\infty, \quad 0<\phi_2(x_0)<+\infty.
\]
Recall that if $(\phi_2,\psi_2)$ is a solution of System \eqref{SchrodingerSystem}, then so is $(\tilde{\phi}_2,\tilde{\psi}_2) = (k \phi_2, \frac{1}{k} \psi_2)$ for $k \neq 0$. Setting $k = \frac{\phi_1(x_0)}{\phi_2(x_0)}$, one has that 
\[
\tilde{\phi}_2(x_0) = \frac{\phi_1(x_0)}{\phi_2(x_0)} \phi_2(x_0) = \phi_1(x_0).
\] 
Thus, without loss of generality, one can always pick two solutions $(\phi_1,\psi_1)$, $(\phi_2,\psi_2)$ where the $\varphi$ agree at one point $x_0\in\I^1$, so that:
\[0<\varphi_1(x_0)=\varphi_2(x_0)<+\infty.
\]
Let $A_1 = \{ x \in \I^1 | \omega_1(x) > 0\}$. On $A_1$, we define
\[h_1(x)=\frac{\omega_1(x)}{\varphi_1(x)}, \quad h_2(x)=\frac{\omega_2(x)}{\varphi_2(x)}.
\]
Then, $h_1$ and $h_2$ are two distinct solutions of  equation (\ref{eq:h}). Let, as before,
\[G(H,y)=\int_{A_1}g(z,y)\frac{\omega_1(z)}{H(z)}dz. 
\]
Then, on $A_2 = \{ y \in \I^2 | \omega_2(y) > 0\}$, we have
\[\psi_1(y)=\frac{\omega_2(y)}{G(h_1,y)}, \quad \psi_2(y)=\frac{\omega_2(y)}{G(h_2,y)}.
\]
From this we deduce that $G(h_1,y)$ and $G(h_2,y)$ are a.e. finite on $A_2$. Let $h(x)=\max(h_1,h_2)$. Then $G(h,y)$ is a.e. finite on $A_2$ and 
\[0<h(x_0)=h_1(x_0)=h_2(x_0)<+\infty.
\]
Let us set
\[h'(x)=\int_{\I^2}g(x,y)\frac{\omega_2 dy}{G(h,y)}.
\]
By the same argument used to prove \ref{lemma:classC} iv), it follows that
\begin{equation}\label{normal}
\int_{A_1}\frac{h'}{h}\omega_1 dx=1.
\end{equation}
Since $h\ge h_1$, it follows from Proposition \ref{prop:MonotoneOmega} that also $h'\ge h_1$. Similarly, $h\ge h_2$ implies $h'\ge h_2$. We infer that $h'\ge h$. It then follows from (\ref{normal}) that, a.e. on $A_1$,
\[h'(x)=h(x).
\]
From $h\ge h_1$, it follows that $G(h,y)\le G(h_1,y)$. Moreover,
\[\int_{\I^2}g(x_0,y)\frac{\omega_2(y)dy}{G(h,y)}=h(x_0)=h_1(x_0)=\int_{\I^2}g(x_0,y)\frac{\omega_2(y)dy}{G(h_1,y)}.
\]
Thus, a.e. on $A_2$, we have
\[G(h,y)=G(h_1,y).
\]
We conclude that everywhere on $\I^1$ we have $h=h_1$, Similarly, we get $h=h_2$ and, finally, $h_1=h_2$ everywhere. 
\end{proof}
\noindent
The following  remark   appears as Remark I on p. 104 of \cite{For}:
\begin{remark} All the results of this paper hold with minor modifications of the statements if one merely assumes that $\omega_1$ and $\omega_2$ are measurable integrable functions.
\end{remark}
We add that if the densities are measurable integrable functions which are bounded on every compact set (this is needed in the proof of the \textbf{Claim} p.16), all proofs can be extended without any modification.

\section{Comparison with the Approach Based on Contracting the Hilbert Metric}\label{Hilbert}
Let us start by observing that Theorem \ref{thm:uniq} asserts that the solution pair is unique up to multiplying $\varphi$ by a positive constant and dividing $\psi$ by the same constant. Moreover, the functions are non-negative. Thus, we only have uniqueness of the {\em ray} in a suitable function space {\em cone}. It is then apparent that projective geometry provides a most natural framework to study convergence of iterative methods. 
A crucial contractivity result that permits to establish existence of solutions  of equations on cones  was proven by Garrett Birkhoff  in 1957 \cite{birkhoff1957extensions}. Important extensions of Birkhoff's result to nonlinear maps were provided by Bushell \cite{bushell1973projective,bushell1973hilbert}. 

Besides the celebrated Perron-Frobenius theorem \cite{birkhoff1962Perron-Frobenius}, various other applications of the Birkhoff-Bushell result have been developed such as to positive integral operators and to positive definite matrices \cite{bushell1973hilbert,lemmens2013birkhoff}. More recently, this geometry has proven useful in various problems concerning  communication and computations over networks (see \cite{tsitsiklis1986distributed} and the work of Sepulchre and collaborators \cite{sepulchre2010consensus,sepulchre2011contraction} on consensus in non-commutative spaces and metrics for spectral densities) and in statistical quantum theory \cite{reeb2011hilbert}. A recent survey on the applications in analysis is \cite{lemmens2013birkhoff}. The use of the projective Hilbert metric is crucial in the nonlinear Frobenius-Perron theory \cite{lemmens2012nonlinear}.

Taking advantage of the Birkhoff-Bushell results on contractivity of linear and nonliner maps on cones,  it was shown in \cite{GP} that the Schr\"odinger bridge for Markov chains and quantum channels can be efficiently obtained from the fixed-point of a map which is contractive in the Hilbert metric. This result extended \cite{FL} which deals with scaling of nonnegative matrices. In \cite{CGP9}, it was shown that a similar approach can be taken in the context of diffusion processes leading to i)  a new proof of a classical result on SBP and ii) providing an efficient computational scheme for both, SBP and OMT. This new approach can be effectively employed, for instance, in image interpolation.

Following \cite{bushell1973hilbert}, we recall some basic concepts and results of this theory.
Let $\cS$ be a real Banach space and let $\cK$ be a closed solid cone in $\cS$, i.e., $\cK$ is closed with nonempty interior ${\rm int}\cK$ and is such that $\cK+\cK\subseteq \cK$, $\cK\cap -\cK=\{0\}$ as well as $\lambda \cK\subseteq \cK$ for all $\lambda\geq 0$. Define the partial order
\[
x\preceq y \Leftrightarrow y-x\in\cK,\quad x< y \Leftrightarrow y-x\in{\rm int}\cK
\]
and for $x,y\in\cK_0:=\cK\backslash \{0\}$, define $M(x,y):=\inf\, \{\lambda\,\mid x\preceq \lambda y\}$, $m(x,y):=\sup \{\lambda \mid \lambda y\preceq x \}$.
Then, the Hilbert metric is defined on $\cK_0$ by
\[
d_H(x,y):=\log\left(\frac{M(x,y)}{m(x,y)}\right).
\]
Strictly speaking, it is a {\em projective} metric since it is invariant under scaling by positive constants, i.e.,
$d_H(x,y)=d_H(\lambda x,\mu y)=d_H(x,y)$ for any $\lambda>0, \mu>0$ and $x,y\in{\rm int}\cK$. Thus, it is actually a  distance between rays. If $U$ denotes the unit sphere in $\cS$, $\left({\rm int}\cK\cap U,d_H\right)$ is a metric space.
\begin{example}\label{ex1} Let $\cK=\R^n_+=\{x\in\R^n: x_i\ge 0\}$ be the positive orthant of $\R^n$. Then, for $x,y\in{\rm int}\R^n_+$, namely with all positive components,
\[M(x,y)=\max_i\{x_i/y_i\}, \quad m(x,y)=\min_i\{x_i/y_i\},
\]
and
\[d_H(x,y)=\log\max_{ij}\{x_iy_j/y_ix_j\}.
\]
\end{example}
Another very important example for applications in many diverse areas of statistics, information theory, control,etc. is the cone of Hermitian, positive semidefinite matrices.
\begin{example}\label{ex2}
Let $\cS=\{X=X^\dagger\in\C^{n\times n}\}$, where $\dagger$ denotes here transposition plus conjugation and, more generally, adjoint. Let $\cK=\{X\in\cS: X\ge 0\}$ be the positive semidefinite matrices. Then, for $X,Y\in{\rm int}\cK$, namely positive definite, we have
\[d_H(X,Y)=\log\frac{\lambda_{\max}\left(XY^{-1}\right)}{\lambda_{\min}\left(XY^{-1}\right)}=\log\frac{\lambda_{\max}\left(Y^{-1/2}XY^{-1/2}\right)}{\lambda_{\min}\left(Y^{-1/2}XY^{-1/2}\right)}.
\]
It is closely connected to the Riemannian (Fisher-information) metric
\begin{eqnarray}\nonumber d_R(X,Y)&=&\|\log\left(Y^{-1/2}XY^{-1/2}\right)\|_{F}\\&=&
\sqrt{\sum_{i=1}^n[\log\lambda_i\left(Y^{-1/2}XY^{-1/2}\right)]^2}.\nonumber
\end{eqnarray}
\end{example}

A map $\cE:\cK\rightarrow\cK$ is called {\em non-negative}. It is called {\em positive} if $\cE:{\rm int}\cK\rightarrow{\rm int}\cK$. If $\cE$ is positive and $\cE(\lambda x)=\lambda^p\cE(x)$ for all $x\in{\rm int}\cK$ and positive $\lambda$, $\cE$ is called {\em positively homogeneous of degree $p$} in ${\rm int}\cK$.
For a positive map $\cE$, the {\em projective diameter} is befined by
\begin{eqnarray*}
\Delta(\cE):=\sup\{d_H(\cE(x),\cE(y))\mid x,y\in {\rm int}\cK\}
\end{eqnarray*}
and the {\em contraction ratio} by
\begin{eqnarray*}
k(\cE):=\inf\{\lambda: \mid d_H(\cE(x),\cE(y))\leq \lambda d_H(x,y),\forall x,y\in{\rm int}\cK\}.
\end{eqnarray*}
Finally, a map $\cE:{\cS}\rightarrow\cS$ is called {\em monotone increasing} if $x\le y$ implies $\cE(x)\le\cE(y)$.
\begin{theorem}[\cite{bushell1973hilbert}] \label{poshom}Let $\cE$ be a monotone increasing positive mapping which is positive
homogeneous of degree $p$ in ${\rm int}\cK$. Then the contraction $k(\cE)$ does not exceed $p$. In particular, if $\cE$ is a positive linear mapping, $k(\cE)\le1$.
\end{theorem}

\begin{theorem}[\cite{{birkhoff1957extensions},bushell1973hilbert}]\label{BBcontraction}
Let $\cE$ be a positive linear map. Then
\begin{equation}\label{condiam}
k(\cE)=\tanh\left(\frac{1}{4}\Delta(\cE)\right).
\end{equation}
\end{theorem}

\begin{theorem}[\cite{bushell1973hilbert}]\label{BBcontraction2}
Let $\cE$ be either
\begin{description}
\item  [a.] a monotone increasing positive mapping which is positive homogeneous of degree $p$  $(0<p<1)$ in ${\rm int}\cK$, or
\item [b.] a positive linear mapping with finite projective diameter.
\end{description}
Suppose the metric space $Y=\left({\rm int}\cK\cap U,d_H\right)$ is complete. Then, in case $(a)$ there exists a unique $x\in{\rm int}\cK$ such that $\cE(x)=x$, in case $(b)$ there exists a unique positive eigenvector of $\cE$ in $Y$.
\end{theorem}
Notice that in both Examples \ref{ex1} and \ref{ex2}, the space $Y=\left({\rm int}\cK\cap U,d_H\right)$ is indeed complete \cite{bushell1973hilbert}.

If we try to use a similar approach to prove existence of the Schr\"odinger system \eqref{SchrodingerSystem}, we may expect that in an infinite dimensional setting questions of boundness or integrability might become a problem. The main difficulty, however, lies here with two other issues. To introduce them, let us observe that in the Birkhoff-Bushell theory we have linear or nonlinear iterations which remain in the {\em interior of a cone}. For example, in the application of the Perron-Frobenius theorem to the ergodic theory of Markov chains, the assumption that there exists a power of the transition matrix with all strictly positive entries ensures that the evolution of the probability distribution occurs in the  {\em interior} of the positive orthant (intersected with the simplex). The first difficulty is that the natural function space cones such as $L^1_+$ ($L^2_+$), namely integrable (square integrable) nonnegative functions on $\R^d$, have {\em empty interior}\,!  The second difficulty is that, even if manage to somehow define a suitable function space cone with nonempty interior, as observed right after Proposition \ref{prop:MonotoneOmega}, the nonlinear map $\Omega$ defined in \eqref{eq:h'} cannot map the interior of the cone to itself since it does not preserve class (C). On the positive side, Proposition \ref{prop:MonotoneOmega} tells us that $\Omega$ is monotone and positively homogenous of degree one on the class (C). In \cite{CGP9}, a cone of nonnegative functions with nonempty interior was indeed defined. Precisely to overcome the second difficulty, however, the two marginals had to be assumed with compact support.

\section{Conclusions}\label{sec:conclusion}
Monotonicity  has been largely used in the literature on Sinkhorn algorithms since Richard Sinkhorn himself \cite{Sin64} down to some recent efficient variants, see e.g. \cite{Schmitzer}. As another example, consider the assignement problem in economics \cite{GKW}. It is there claimed, citing a future publication, that monotonicity together with Tarsky's fixed point theorem allows to establish existence for a {\em nonlinear} Schr\"odinger system. Nevertheless, all of these algorithms deal with the discrete, finite setting. To fully appreciate Fortet's algorithm, understanding the crucial difficulties he was able to get around, we need to compare his approach to those based on contracting a projective metric as done in the last part of Section \ref{Hilbert}.

A careful reading of the proof of Theorem \ref{thm:ex1}, shows that Fortet's iteration either stops after a finite number of steps in a fixed point of the map $\Omega$, or the $H'_n$ converge to an {\em everywhere positive} function $H'$ which is the fixed point of $\Omega$. Also observe that in the approximation scheme (\ref{approxscheme}), the min step serves to provide an upper bound and the max serves to render the function bounded away from zero. In view of all of this, it might  be possible to interpret Fortet's approach in a projective geometry setting. The max step, however, makes so that the functions $H'_n$ are produced by a composition of {\em different} maps and therefore fixed-point arguments are out of the question. 

Condition (\ref{cond:thm1}) of Theorem \ref{thm:ex1} expresses a rather delicate relation between the kernel $g(\cdot,\cdot)$ and the two given marginals. This condition appears to be not very restrictive and totally original: It seems in fact more general than available existence conditions such as \cite[Proposition 2.5]{leo}.

Finally, it would be nice to apply Fortet's ingenious method of successive approximations to other related problems such as the (regularized) optimal transport barycenter problem, see e.g. \cite{BCCNP,PC}. This, however, will be considered elsewhere.

\appendix
\section*{Appendix:} 
\subsection*{Proof of  \eqref{claim:measure} from Theorem \ref{thm:ex1}} \label{proof:claimmeasure}
Let $Z \subset \I^2$ be the set of $\{ y \in \I^2 | g(x_0,y) = 0 \}$.

Define $Z_k = \{ y \in \I^2 |\; g(x_0,y) < \frac{1}{k} \}$ for $k \in \mathbb{N}^*$. We have $Z_{k+1} \subset Z_k$, and $Z_k \downarrow Z$ as $k \to +\infty$.

By assumption \ref{Hvi} we know that $Z$ has Lebesgue measure $0$.
From the continuity of $g$ \ref{Hiv}, we also know that $Z$ is closed. 

Hence $m(Z_k) \to 0$ as $k \to +\infty$.

Denote by $\I^2_k = \I^2 \backslash Z_k$. Then we have $\I^2_k \subset \I^2_{k+1}$ and $\I^2_k \uparrow \I^2 \backslash Z$ as $k \to +\infty$.\\

Since 
\[
H_n'(x_0) = \int_{\I^2} g(x_0,y) \frac{\omega_2(y)}{G(H_n,y)} dy \to 0, \quad \text{ as } n \to +\infty
\]

$\forall \epsilon >0$, we have for $n$ large enough:
\[
\int_{\I^2} g(x_0,y) \frac{\omega_2(y)}{G(H_n,y)} dy < \epsilon
\]

Fix $\epsilon >0, k \in \mathbb{N}^*$.
We then have for $n$ large enough:
\[
0 \leq \int_{\I^2_{k}} g(x_0,y) \frac{\omega_2(y)}{G(H_n,y)} dy + \int_{\I^2\backslash\I^2_{k}} g(x_0,y) \frac{\omega_2(y)}{G(H_n,y)} dy < \epsilon
\]

and in particular, by non-negativity, the first integral yields:

\[
0 \leq \int_{\I^2_{k}} \frac{\omega_2(y)}{G(H_n,y)} dy  < k \epsilon 
\]
This implies that the measure $\frac{\omega_2(y)}{G(H_n,y)} dy$ converges weakly to $0$ on $\I^2_{k}$. Indeed, it is the case when evaluated on any step function with support included in $\I^2_{k}$, and step functions are dense in the family of bounded continuous functions.\\

We would like the measure $\frac{\omega_2(y)}{G(H_n,y)} dy$ to converge to $0$ for any step function whose support $I$ is included in $\I^2$, and not merely on $\I^2_k$.

Pick a subset $I \subset{\I^2}$, and consider:
\[
\int_{\I^2} \1_{I}(y) \frac{\omega_2(y)}{G(H_n,y)} dy  = \int_{\I^2_{k} \cap I} \frac{\omega_2(y)}{G(H_n,y)} dy + \int_{(\I^2_{k} \cap I)^C} \frac{\omega_2(y)}{G(H_n,y)} dy 
\]
The first integral converges to $0$ as $n \to +\infty$, since the measure $\frac{\omega_2(y)}{G(H_n,y)} dy$ converges weakly to $0$ on $\I^2_{k}$.

As for the second integral, we have that $H_n \leq H_1$, so $\frac{\omega_2(y)}{G(H_n,y)} \leq \frac{\omega_2(y)}{G(H_1,y)}$ which implies:
\[
\int_{(\I^2_{k} \cap I)^C} \frac{\omega_2(y)}{G(H_n,y)} dy  \leq \int_{(\I^2_{k} \cap I)^C} \frac{\omega_2(y)}{G(H_1,y)} dy \leq \int_{Z_{k}} \frac{\omega_2(y)}{G(H_1,y)} dy  
\]

where the last inequality comes from $(\I^2_{k} \cap I)^C \subset Z_{k}$.

Condition \eqref{cond:thm1} states that $\int_{\I^2} \frac{\omega_2(y)}{G(H_1,y)} dy < +\infty$, thus we know that the measure $\frac{\omega_2(y)}{G(H_1,y)} dy$ is absolutely continuous with respects to the Lebesgue measure $m$ on $\I^2$. This implies that the second integral converges to $0$, as $k \to +\infty$ since $m(Z_{k}) \to 0$.

Hence, for any measurable $I \subset \I^2$, $\int_{\I^2} \1_{I}(y) \frac{\omega_2(y)}{G(H_n,y)} dy \to 0$ as $n \to +\infty$. 

\paragraph{Acknowledgements}
The authors thank Robert V. Kohn for useful suggestions. The second named author would also like to thank the Courant Institute of Mathematical Sciences of the New York University for the hospitality during the time this paper was written. The authors finally wish to thank two anonymous reviewers for a very careful reading and providing plenty of general and specific comments/suggestions on how to improve the paper. The second named author was partly supported by the University of Padova Research Project CPDA 140897.

\bibliographystyle{spmpsci_unsrt}
\bibliography{fortet}

\end{document}